\documentclass[gpscopy,onehalfspacing,12pt]{ubcdiss}

\usepackage{times,mathptmx,courier}
\usepackage[scaled=.92]{helvet}

\usepackage{amsmath, amsfonts, amssymb, accents, mdwlist}

\DeclareMathAlphabet{\mathcal}{OMS}{cmsy}{m}{n}

\usepackage{amsthm}
\theoremstyle{plain}
\newtheorem{theorem}{Theorem}
\newtheorem{lemma}[theorem]{Lemma}
\newtheorem{corollary}[theorem]{Corollary}

\theoremstyle{definition}
\newtheorem*{remark}{Remark}
\newtheorem*{remarks}{Remarks}
\newtheorem*{example}{Example}

\usepackage{algorithm}
\usepackage[noend]{algpseudocode}

\algblockdefx{MRepeat}{EndRepeat}{\textbf{Repeat}}{}
\algnotext{EndRepeat}

\algblockdefx{MForAll}{EndForAll}{\textbf{For all}}{}
\algnotext{EndForAll}

\usepackage{checkend} % better error messages on left-open environments
\usepackage{graphicx} % for incorporating external images
\usepackage{booktabs} % Provides commands for typesetting tables.
\usepackage{comment} % Comment environment.

\usepackage{listings} % Supports for source doe listings.
\usepackage[printonlyused,nohyperlinks]{acronym} % Acronyms and Glossaries
\usepackage{color, xcolor}
\definecolor{greytext}{gray}{0.5}
\definecolor{crimsonred}{RGB}{132,22,23}
\definecolor{darkblue}{RGB}{72,61,139}

\usepackage[numbers,sort&compress]{natbib}

\usepackage[compact]{titlesec}
\titleformat*{\section}{\singlespacing\raggedright\bfseries\Large}
\titleformat*{\subsection}{\singlespacing\raggedright\bfseries\large}
\titleformat*{\subsubsection}{\singlespacing\raggedright\bfseries}
\titleformat*{\paragraph}{\singlespacing\raggedright\itshape}

\usepackage[format=hang,indention=-1cm,labelfont={bf},margin=1em]{caption}

\usepackage{url}
\usepackage{pdfpages}	% insert pages from other PDF files
\usepackage{longtable}	% provide tables spanning multiple pages
\usepackage{chngpage}	% for changing the page widths on demand
\usepackage{tabularx}   % enhanced tabular environment
\usepackage{subfig}     % include subfigures in figures.

\usepackage{enumitem}   % allows pausing and resuming enumerate environments.
\setlist[enumerate]{label*={\normalfont(\Alph*)},ref=(\Alph*)}

\usepackage[bookmarks,bookmarksnumbered,%
    pagebackref,linktocpage%
    ]{hyperref}
\hypersetup{
    colorlinks=true,
    citecolor=black,
    filecolor=black,
    linkcolor=black,
    urlcolor=black,
    linktoc=all
}
\newcommand{\nocitations}{\relax}
\renewcommand*{\backrefalt}[4]{%
\textcolor{greytext}{\ifcase #1%
\nocitations%
\or
\(\rightarrow\) page #2%
\else
\(\rightarrow\) pages #2%
\fi}}

\DeclareUrlCommand\DOI{}
\newcommand{\doi}[1]{\href{http://dx.doi.org/#1}{\DOI{doi:#1}}}


\DeclareMathOperator{\minkdim}{\dim_{\mathbf{M}}}
\DeclareMathOperator{\hausdim}{\dim_{\mathbf{H}}}
\DeclareMathOperator{\lowminkdim}{\dim_{\underline{\mathbf{M}}}}
\DeclareMathOperator{\upminkdim}{\dim_{\overline{\mathbf{M}}}}
\DeclareMathOperator{\fordim}{\dim_{\mathbf{F}}}

\DeclareMathOperator{\lmbdim}{\dim_{\underline{\mathbf{MM}}}}
\DeclareMathOperator{\umbdim}{\dim_{\overline{\mathbf{MM}}}}

\DeclareMathOperator{\RR}{\mathbf{R}}
\DeclareMathOperator{\ZZ}{\mathbf{Z}}

\DeclareMathOperator{\QQ}{\mathbf{Q}}

\DeclareMathOperator{\EE}{\mathbf{E}}
\DeclareMathOperator{\PP}{\mathbf{P}}
\DeclareMathOperator{\prob}{\Prob}
\DeclareMathOperator{\Prob}{\mathbf{P}}
\DeclareMathOperator{\expect}{\Expect}
\DeclareMathOperator{\Expect}{\mathbf{E}}

\DeclareMathOperator{\AAA}{\mathbf{A}}

\DeclareMathOperator{\B}{\mathcal{B}}
\DeclareMathOperator{\C}{\mathcal{C}}

\DeclareMathOperator{\DQ}{\mathcal{Q}}

\DeclareMathOperator{\DD}{\mathcal{D}}
\DeclareMathOperator{\DR}{\mathcal{R}}

\DeclareMathOperator{\Config}{\mathcal{C}}
\DeclareMathOperator{\diam}{\text{diam}}
\DeclareMathOperator{\divides}{\mid}

\DeclareMathOperator{\setcolon}{\colon}

\title{Cartesian Products Avoiding Patterns}
\author{Jacob Denson}
\degreetitle{Master of Science}

\institution{The University of British Columbia}
\campus{Vancouver}

\faculty{The Faculty of Graduate and Postdoctoral Studies}
\department{Mathematics}
\submissionmonth{December}
\submissionyear{2019}

\examiningcommittee{Malabika Pramanik, Mathematics}{Co-Supervisor}
\examiningcommittee{Joshua Zahl, Mathematics}{Co-Supervisor}

\hypersetup{
  pdftitle={Cartesian Products Avoiding Patterns},
  pdfauthor={Jacob Denson},
  pdfkeywords={Geometric Measure Theory, Pattern Avoidance, Harmonic Analysis}
}

\begin{document}

% Thesis Guidelines available at:
% 	http://www.grad.ubc.ca/current-students/dissertation-thesis-preparation/order-components

% 1. Title page (mandatory)
\maketitle

% 2. Committee page (mandatory): lists supervisory committee and if applicable, the examining committee
\makecommitteepage

% 3. Abstract (mandatory - maximum 350 words)
%% The following is a directive for TeXShop to indicate the main file
%%!TEX root = diss.tex

\chapter{Abstract}

% MAXIMUM 350 WORDS!

The pattern avoidance problem seeks to construct a set with large fractal dimension that avoids a prescribed pattern, such as three term arithmetic progressions, or more general patterns such as finding a set whose Cartesian product avoids the zero set of a given function. Previous work on the subject has considered patterns described by polynomials, or functions satisfying certain regularity conditions. We provide an exposition of some results in this setting, as well as considering new strategies to avoid `rough patterns'. There are several problems that fit into the framework of rough pattern avoidance. For instance, we prove that for any set $X$ with lower Minkowski dimension $s$, there exists a set $Y$ with Hausdorff dimension $1-s$ such that for any rational numbers $a_1, \dots, a_N$, $a_1Y + \dots + a_NY$ is disjoint from $X$, or intersects solely at the origin. As a second application, we construct subsets of Lipschitz curves with dimension $1/2$ not containing the vertices of any isosceles triangle.
\cleardoublepage

% 4. Lay Summary (Effective May 2017, mandatory - maximum 150 words)
%% The following is a directive for TeXShop to indicate the main file
%%!TEX root = diss.tex

%% https://www.grad.ubc.ca/current-students/dissertation-thesis-preparation/preliminary-pages
%% 
%% LAY SUMMARY Effective May 2017, all theses and dissertations must
%% include a lay summary.  The lay or public summary explains the key
%% goals and contributions of the research/scholarly work in terms that
%% can be understood by the general public. It must not exceed 150
%% words in length.

%% The lay or public summary explains the key goals and contributions of
%% the research\slash{}scholarly work in terms that can be understood by the
%% general public. It must not exceed 150 words in length.

\chapter{Lay Summary}

%Imagine a patch of carpet under a microscope. Zooming in, we see the carpet is really a collection of twines tied together. On closer inspection, those twines break off into smaller tufts of fabric. The carpet is rough at all scales. Shapes like circles or polygons do not have complexity at arbitrarily small scales. To model the roughness of a carpet, you'd need a fractal: a shape with complex structure at all small scales. Such models are often useful in small-scale physics or computer graphics.

%It is easy to construct polyhedra with geometric properties at a single scale. But it is non-trivial to construct fractals with properties at many scales. For instance, how do we construct a fractal which intersects any line in at most two points? In this thesis, we begin with an exposition on previous constructions in the literature, and then provide new construction techniques utilizing randomness.

Geometers are often interested in constructing shapes satisfying certain properties. For instance, given three points, can one find a circle connecting them? Most questions of this type involving shapes like circles or polygons have been answered. But many open questions remain about more modern families of shapes. Here, we focus on \emph{fractals}, a class of shapes whose most well known representatives include the Koch snowflake and the Sierpinski triangle. Fractals often occur in applications such as small scale physics and computer graphics.

This thesis focuses on constructing \emph{large} fractals which avoid the existence of certain configurations. For example, can one construct a large fractal so that one cannot form an equilateral triangle from three points contained on the fractal? We begin with an exposition of some previous results of this type which have been achieved in the literature, and then provide new construction techniques utilizing a novel random approach.
\cleardoublepage

% 5. Preface
%% The following is a directive for TeXShop to indicate the main file
%%!TEX root = diss.tex

\chapter{Preface}

This thesis gives an exposition by the author, of the pattern avoidance problem and the geometric measure theory required to understand the pattern avoidance problem in the non-discrete setting. In Chapter \ref{ch:RoughSets} and \ref{ch:Applications}, the author presents details of joint work with his supervisors Dr. Joshua Zahl and Dr. Malabika Pramanik. The results of these sections have been accepted for publication in the Springer series \emph{Harmonic Analysis and Applications}. As is the norm in mathematical research, all researchers are assumed to have contributed equally to these results. But to list concrete contributions, the author of this thesis reviewed the background literature detailed in the bibliography to this paper, and came up with the main problem statement behind Theorem \ref{mainTheorem}. In Chapter \ref{ch:Conclusions}, the author presents details on partially completed results emerging from discussions with Dr. Joshua Zahl and Dr. Malabika Pramanik, which he hopes can be refined and published in the near future.
\cleardoublepage

% 6. Table of contents (mandatory - list all items in the preliminary pages
% starting with the abstract, followed by chapter headings and
% subheadings, bibliographies and appendices)
\tableofcontents
\cleardoublepage	% required by tocloft package

% 7. List of tables (mandatory if thesis has tables)
%\listoftables
%\cleardoublepage	% required by tocloft package

% 8. List of figures (mandatory if thesis has figures)
%\listoffigures
%\cleardoublepage	% required by tocloft package

% 9. List of illustrations (mandatory if thesis has illustrations)
% 10. Lists of symbols, abbreviations or other (optional)

% 11. Glossary (optional)
%\input{glossary}	% always input, since other macros may rely on it

\textspacing		% begin one-half or double spacing

% 12. Acknowledgements (optional)
%% The following is a directive for TeXShop to indicate the main file
%%!TEX root = diss.tex

\chapter{Acknowledgments}

I am indepted to my advisors, Dr. Malabika Pramanik and Dr. Josh Zahl, for their key insights in the past two years of research. Their advice will help me conduct research for years to come. Without their tough scrutiny of my writing style over the past year, this thesis would be exponentially less legible.

I would also like to give credit to Dr. Zachary Friggstad. Our long discussions have changed the way I think about mathematics, and I look forward to many more. His influence is felt throughout the new techniques developed in this thesis.

Thanks to my family, for their support throughout my education. I would especially like to thank my grandfather, Ted McClung. Without your encouragement in my early years of undergraduate education, it is unlikely I would have found my passion for higher mathematics.

Finally, I'd like to thank the UBC mathematics department, and greater student community, for keeping me grounded during many stressful moments over the past two years.

% It is important to acknowledge the University of British Columbia, and the NSERC research fund, for their financial support during my masters degree.

% 13. Dedication (optional)

% Body of Thesis (not all sections may apply)
\mainmatter

\acresetall	% reset all acronyms used so far

% 1. Introduction
%% The following is a directive for TeXShop to indicate the main file
%%!TEX root = diss.tex

\chapter{Introduction}
\label{ch:Introduction}

In this thesis, we study a simple family of questions:
\begin{center}
	{\it How large can Euclidean sets be not containing geometric patterns?}
\end{center}
Consider two examples:
\begin{itemize}
    \item How large can a set $X \subset \RR^d$ be that contains no three collinear points?

    \item What is the maximal size of a set $X \subset \RR^d$ not containing the vertices of an isosceles triangle?
\end{itemize}
Aside from pure geometric interest, these problems provide useful settings to test methods of ergodic theory, additive combinatorics and harmonic analysis.

Sets which avoid the patterns we consider are highly irregular, in many ways behaving like a fractal set. Our understanding of their structure requires techniques from geometric measure theory. In particular, we use various \emph{fractal dimensions} to measure the size of a pattern avoiding set.

%In particular, we use the Hausdorff dimension and Minkowski dimension, as well as various other more modern variants of fractal dimension which give more structural information about a set.

%The Lebesgue density theory shows any subset of $\RR^d$ with positive measure must contain a copy of the vertices of any suitably small isosceles triangle. Similar techniques show sets avoiding any of the patterns considered in this thesis must have measure zero. Thus non-trivial pattern avoiding sets must be highly irregular, and require techniques from geometric measure theory to be understood.

%In particular, the Lebesgue measure cannot quantify the `maximal size' of a pattern avoiding set. Thus a `second-order' quantification of the size of a measure zero must be introduced, and this is satisfied by the fractional dimension of a set, first introduced by Felix Hausdorff in 1918. We shall take the Hausdorff dimension he introduced as a primary measure of a set's size.%, as well as various other more modern notions of fractional dimension which give more structural information about a set.

At present, many fundamental questions about geometric structure and its relation to fractal dimension remain unsolved. One might expect sets with sufficiently large fractal dimension must contain a given pattern. Evidence is provided by Theorem 6.8 of \cite{Matilla}, which shows that any set $X \subset \RR^d$ with Hausdorff dimension exceeding one must contain three collinear points. But our expectation is not always true. For instance, Theorem 2.3 of \cite{Maga} constructs a set $X \subset \RR^2$ with full Hausdorff dimension such that no four points in $X$ form the vertices of a parallelogram. Thus the conjecture that sufficiently large sets must contain a given pattern depends on the particular patterns involved. For most geometric configurations, it remains unknown at what threshold patterns are guaranteed, or whether such a threshold exists at all.

%So finding sets with large Hausdorff dimension avoiding patterns is an important topic to determine for which classes of patterns our intuition remains true.

Our goal in this thesis is to derive new methods for constructing sets with large dimension avoiding patterns. In particular, we expand on a number of general {\it pattern dissection methods} which have proven useful in the area, originally developed by Keleti but also studied notably by Fraser, Math\'{e}, and Pramanik. In Chapter \ref{ch:RelatedWork}, we give an exposition of some of their methods, after we establish some background in Chapter \ref{ch:Background}. Chapter \ref{ch:RoughSets} provides our main contribution, avoiding new classes of patterns using random dissection methods. This enables us to expand the utility of pattern dissection methods from regular families of patterns to a family of {\it fractal avoidance problems}, that previous methods were completely unavailable to address. Such problems include finding large sets $X \subset \RR^d$ such that $X + X$ avoids a set $Y$, where $Y$ has a fixed fractal dimension, which we discuss in Chapter \ref{ch:Applications}. Chapter \ref{ch:Conclusions} describes various improvements to the results of Chapter \ref{ch:RoughSets} we hope to develop in the future.

% 2. Main body
%% The following is a directive for TeXShop to indicate the main file
%%!TEX root = diss.tex

\chapter{Background}
\label{ch:Background}

In this chapter we discuss the required background to understand the techniques of the pattern avoidance problem. The majority of the background in geometric measure theory can be found in other resources, e.g. in \cite{Falconer}, \cite{MattilaGeomMeasure}, or \cite{TaoHausdorff}, though not in the context of the pattern avoidance problem.

\section{Configuration Avoidance}

Our main focus in this thesis is the \emph{pattern avoidance problem}. In this section, we formalize the notion of a pattern, and what it means to avoid it. Given a set $\mathbf{A}$, we let
\[ \Config^n(\AAA) = \{ (a_1, \dots, a_n) \in \AAA^n: a_i \neq a_j\ \text{if $i \neq j$} \}. \]
and
\[ \Config(\AAA) = \bigcup_{n = 1}^\infty \Config^n(\AAA). \]
We call $\Config(\AAA)$ the \emph{configuration space} of $\mathbf{A}$.

\begin{example}[Non-Colinearity]
	We say a set $X \subset \RR^d$ \emph{avoids colinear points} if no three points $x_1,x_2,x_3 \in X$ lie on a common line in $\RR^d$. Define
	\[ \C = \left\{ (x, x + av, x + 2av) \in \C^3(\RR^d) : a \in \RR - \{ 0 \}, v \in \RR^d - \{ 0 \} \right\}. \]
	Then $X$ avoids colinear points if and only if $\C^3(X)$ is disjoint from $\C$.
\end{example}

\begin{example}[Isosceles Triangle Configuration]
	We say a set $X \subset \RR^2$ \emph{avoids isosceles triangles} if no three points $x_1,x_2,x_3 \in X$ form the vertices of a non-degenerate isosceles triangle. Define
	\[ \C = \left\{ (x_1, x_2, x_3) \in \Config^3(\RR^2) : |x_1-x_2| = |x_1-x_3| \right\}. \]
	A set $X$ avoids isosceles triangles if and only if $\C^3(X)$ is disjoint from $\C$. %Notice that $|x_1 - x_2| = |x_1 - x_3|$ holds if and only if $|x_1 - x_2|^2 = |x_1 - x_3|^2$, which is an algebraic equation in the coordinates of $x_1,x_2$, and $x_3$. Thus $\C$ is an algebraic hypersurface of degree two in $\RR^6$.
\end{example}

\begin{example}[Linear Independence Configuration]
	Let $V$ be a vector space over a field $K$. We set
	\[ \C = \bigcup_{n = 1}^\infty \left\{ (x_1, \dots, x_n) \in \Config^n(V): \left\{ \begin{array}{c}
			\text{there is $a_1, \dots, a_n \in K$ such}\\
			\text{that $a_1x_1 + \dots + a_nx_n = 0$}
		\end{array} \right\} \right\}. \]
	A set $X \subset V$ is linearly independent in $V$ if and only if $\C(X)$ is disjoint from $\C$.
\end{example}

\begin{remark}
	In this thesis, we will be most interested in the linear independence configuration where $K = \QQ$, and $V = \RR$. Results in this setting are discussed in both Chapter \ref{ch:Applications} and \ref{ch:Conclusions}.
\end{remark}

\begin{example}[Sum Set Configuration]
	Let $G$ be an abelian group, and fix $Y \subset G$. Set
	\[ \C^1 = \{ g \in \Config^1(G): g + g \in Y \} \quad \text{and} \quad \C^2 = \{ (g_1,g_2) \in \Config^2(G): g_1 + g_2 \in Y \}. \]
	Define $\C = \C^1 \cup \C^2$. Then $(X + X) \cap Y = \emptyset$ if and only if $\C(X)$ is disjoint from $\C$.
\end{example}

All the configurations we discuss in this thesis can be specified in terms of subsets of $\Config(\AAA)$. Thus we formally define a \emph{configuration} on $\AAA$ to be a subset of $\Config(\AAA)$. In particular, if $n > 0$, we say a configuration $\C$ is an \emph{$n$ point configuration} if it is a subset of $\Config^n(\AAA)$. For a fixed configuration $\C$ on $\AAA$, we say a set $X \subset \AAA$ \emph{avoids} $\C$ if $\Config(X)$ is disjoint from $\C$. The \emph{pattern avoidance problem} asks to find sets $X$ of maximal size avoiding a fixed configuration $\C$. Often, the configuration $\C$ describes algebraic or geometric structure, and the pattern avoidance problem asks to find the maximal size of a set before it is guaranteed to have such structure.

%\begin{example}[General Position Configuration]
%	Suppose we wish to find a subset $X$ of $\RR^d$ such that for each positive integer $k \leq d$, and for each collection of $k+1$ distinct points $x_1, \dots, x_{k+1} \in X$, the points do not lie in a $k-1$ dimensional hyperplane. For each $k \leq d$, set
	%
%	\[ \C^{k+1} = \{ (x_0, x_1, \dots, x_k) \in \Config^{k+1}(\RR^d): x_1-x_0, \dots, x_k - x_0\ \text{are linearly dependant} \}. \]
	%
%	If we define $\C = \bigcup_{k = 2}^d \C^k$, then a set $X$ avoids $\C$ precisely when all finite collection of distinct points in $X$ lie in general position. Notice that
	%
%	\[ \C^{k+1} = \bigcup \left\{ \text{span}(y_1, \dots, y_k) \times \{ y \} : y = (y_1, \dots, y_k) \in \Config^k(\RR^d) \right\} \cap \Config^{k+1}(\RR^d). \]
	%
%	so each $\C^{k+1}$ is essentially a union of $k$ dimensional hyperplanes.
%\end{example}

Depending on the structure of the ambient space $\AAA$ and the configuration $\C$, there are various ways of measuring the size of sets $X \subset \AAA$ for the purpose of the pattern avoidance problem. If $\AAA$ is finite, for instance, a natural choice is the cardinality of $X$. But our goal is to study pattern avoidance where $\AAA = \RR^d$. In certain cases, one can use the Lebesgue measure to determine the size of a pattern avoiding set. But this really only works for `discrete' configurations on $\RR^d$, as the next theorem shows, under the often true assumption that $\C$ is \emph{translation invariant}, i.e. that if $(a_1, \dots, a_n) \in \C$ and $b \in \RR^d$, $(a_1 + b, \dots, a_n + b) \in \C$.

\begin{theorem}
	Let $\C$ be a $n$-point configuration on $\RR^d$. Suppose
	\begin{enumerate}
		\item \label{translationinvariance} $\C$ is translation invariant.
		\item \label{nonDiscreteConfig} For any $\varepsilon > 0$, there is $(a_1, \dots, a_n) \in \C$ with $\diam \{ a_1, \dots, a_n \} \leq \varepsilon$.
	\end{enumerate}
	Then no set with positive Lebesgue measure avoids $\C$.
\end{theorem}
\begin{proof}
	Let $X \subset \RR^d$ have positive Lebesgue measure. The Lebesgue density theorem implies that there exists a point $x \in X$ such that
	\begin{equation} \label{densityApplication} \lim_{l(Q) \to 0} \frac{|X \cap Q|}{|Q|} = 1, \end{equation}
	where $Q$ ranges over all axis-oriented cubes in $\RR^d$ with $x \in Q$, and $l(Q)$ denotes the sidelength of $Q$. Fix $\varepsilon > 0$, to be specified later, and choose $r$ small enough that $|X \cap Q| \geq (1 - \varepsilon) |Q|$ for any cube $Q$ with $x \in Q$ and $l(Q) \leq r$. Now let $Q_0$ denote a cube centered at $x$ with $l(Q_0) \leq r$. Applying Property \ref{nonDiscreteConfig}, we find $C = (a_1, \dots, a_n) \in \C$ such that
	\begin{equation} \label{equation690346024} \diam \{ a_1, \dots, a_n \} \leq l(Q_0)/2. \end{equation}
	For each $p \in Q_0$, let $C(p) = (a_1(p), \dots, a_n(p))$, where $a_i(p) = p + (a_i - a_1)$. A union bound shows
	\begin{equation} \label{equation548} \left| \{ p \in Q_0 : C(p) \not \in \C(X) \} \right| \leq \sum_{i = 1}^n \left| \{ p \in Q_0 : a_i(p) \not \in X \} \right|.
	\end{equation}
	We have $a_i(p) \not \in X$ precisely when $p + (a_i - a_1) \not \in X$, so
	\begin{equation} \label{equation1243462}
		|\{ p \in Q_0 : a_i(p) \not \in X \}| = |(Q_0 + (a_i - a_1)) \cap X^c|.
	\end{equation}
	Note $Q_0 + (a_i - a_1)$ is a cube with the same sidelength as $Q_0$. Equation \eqref{equation690346024} implies $|a_i - a_1| \leq l(Q_0)/2$, so $x \in Q_0 + (a_i - a_1)$. Thus \eqref{densityApplication} shows
	\begin{equation} \label{equation543} |Q_0 + (a_i - a_1)) \cap X^c| \leq \varepsilon |Q_0|. \end{equation}
	Combining \eqref{equation548}, \eqref{equation1243462}, and \eqref{equation543}, we find
	\[ \left| \{ p \in Q_0 : C(p) \not \in \C(X) \} \right| \leq \varepsilon n |Q_0|. \]
	Provided $\varepsilon n < 1$, this means there is $p \in Q_0$ with $C(p) \in \C(X)$. Property \ref{translationinvariance} implies $C(p) \in \C$, so $X$ does not avoid $\C$.
\end{proof}

% TO DO: SWAPPING THIS RESULT FOR A DENSITY RESULT, IF WE CAN OBTAIN IT.

Since no set of positive Lebesgue measure can avoid non-discrete, translation invariant configurations, we cannot use the Lebesgue measure to quantify the size of pattern avoiding sets in this setting. Geometric measure theory provides us with various quantities that are able to distinguish between the size of sets of measure zero. These are the \emph{fractal dimensions} of a set. In all configuration avoidance problems in this thesis, we use a fractal dimension to measure the size of configuration avoiding sets.

There are many variants of fractal dimension. Here we choose to focus on Minkowski dimension, Hausdorff dimension, and Fourier dimension. These quantities assign the same dimension to any smooth manifold with non-vanishing curvature, but can differ for rougher sets. Minkowski dimension measures relative density at a single scale, whereas Hausdorff dimension measures relative density at countably many scales. Fourier dimension is a refinement of Hausdorff dimension which places structural constraints on the set in the `frequency domain'.

\section{Minkowski Dimension}

Given $l > 0$, and a bounded set $E \subset \RR^d$, we let $N(l,E)$ denote the \emph{covering number} of $E$, i.e. the minimum number of sidelength $l$ cubes required to cover $E$. We define the \emph{lower} and \emph{upper} Minkowski dimension as
\[ \lowminkdim(E) = \liminf_{l \to 0} \left[ \frac{\log(N(l,E))}{\log(1/l)} \right] \quad\text{and}\quad \upminkdim(E) = \limsup_{l \to 0} \left[ \frac{\log(N(l,E))}{\log(1/l)} \right]. \]
If $\upminkdim(E) = \lowminkdim(E)$, then we refer to this common quantity as the \emph{Minkowski dimension} of $E$, denoted $\minkdim(E)$. Thus $\lowminkdim(E) < s$ if there exists a sequence of lengths $\{ l_k \}$ converging to zero such that for each $k$, $E$ is covered by fewer than $(1/l_k)^s$ sidelength $l_k$ cubes, and $\upminkdim(E) < s$ if $E$ is covered by fewer than $(1/l)^s$ sidelength $l$ cubes for \emph{any} suitably small $l$.

\begin{remark}
	Any cube with sidelength $r$ is covered by $O_d(1)$ balls of radius $r$. Conversely, any ball of radius $r$ is covered by $O_d(1)$ cubes of sidelength $r$. Thus for any $r > 0$, if we temporarily define $N_B(r,E)$ to be the optimal number of radius $r$ \emph{balls} it takes to cover $E$, then $N(r,E) \sim_d N_B(r,E)$. As $r \to 0$, this means
	\[ \frac{\log(N(r,E))}{\log(1/r)} = \frac{\log(N_B(r,E))}{\log(1/r)} + o(1). \]
	In particular, $\lowminkdim(E) < s$ if and only if there exists a sequence of lengths $\{ r_k \}$ such that $E$ is covered by $(1/r_k)^s$ radius $r_k$ \emph{balls}, and $\upminkdim(E) < s$ if and only if $E$ is covered by $(1/r)^s$ radius $r$ \emph{balls}, for any suitably small $r > 0$.
\end{remark}

It is often easy to upper bound the Minkowski dimension of a set, simply by providing a cover of the set and counting the number of cubes that cover it. Let us now consider an example. We say a set $S \subset \RR^d$ is an $s$ dimensional \emph{Lipschitz manifold} if there exists a family of bounded, open subsets $\{ U_\alpha \}$ of $\RR^s$, together with a family of bi-Lipschitz maps $\{ f_\alpha: U_\alpha \to S \}$, such that the sets $\{ f_\alpha(U_\alpha) \}$ form a relatively open cover of $S$. Every $C^1$ manifold in $\RR^d$ is a Lipschitz manifold. This example proves useful in Chapter \ref{ch:RoughSets}.

\begin{theorem} \label{ManifoldDimensionThm}
	Let $S \subset \RR^d$ be a Lipschitz manifold of dimension $s$. Then for any compact set $K \subset S$, $\upminkdim(K) \leq s$.
\end{theorem}
\begin{proof}
	Since $K$ is compact, we can find finitely many bi-Lipschitz maps $f_1, \dots, f_N$ such that the family $\{ f_i(U_i) : 1 \leq i \leq N \}$ covers $K$. Then there is a constant $C > 0$ such that for each $i$, if $x,y \in U_i$,
	\[ |f_i(x) - f_i(y)| \leq C \cdot |x-y|. \]
	Since each $U_i$ is bounded, for any $r > 0$, we can find a family of balls $B_{i,1}, \dots, B_{i,M_i}$ of radius $(r/C)$ covering $U_i$, such that the centers $x_{i,1}, \dots, x_{i,M_i}$ of the balls also lie in $U_i$, and $M_i \lesssim (C/r)^s$. But then the balls of radius $r$ with centers lying in
	\[ \{ f_i(x_{i,j}) : 1 \leq i \leq N, 1 \leq j \leq M_i \} \]
	cover $K$, and there are $\sum_{i = 1}^N M_i \lesssim (N/C^s) r^{-s}$ such balls. Since $C$ and $N$ are independent of $r$, and $r > 0$ was arbitrary, this shows $\upminkdim(K) \leq s$.
\end{proof}

\section{Hausdorff Dimension}

For $E \subset \RR^d$ and $\delta > 0$, we define the \emph{Hausdorff content}
\[ H_\delta^s(E) = \inf \left\{ \sum_{k = 1}^\infty l(Q_k)^s : E \subset \bigcup_{k = 1}^\infty Q_k, l(Q_k) \leq \delta \right\}. \]
The \emph{$s$-dimensional Hausdorff measure} of $E$ is
\[ H^s(E) = \lim_{\delta \to 0} H_\delta^s(E) = \sup \left\{ H^s_\delta(E) : \delta > 0 \right\}. \]
It is easy to see $H^s$ is an exterior measure on $\RR^d$, and $H^s(E \cup F) = H^s(E) + H^s(F)$ if the Hausdorff distance $d(E,F)$ between $E$ and $F$ is positive. So $H^s$ is actually a metric exterior measure, and the Caratheodory extension theorem shows all Borel sets are measurable with respect to $H^s$. Sometimes, it is convenient to use the exterior measure
\[ H^s_\infty(E) = \inf \left\{ \sum_{k = 1}^\infty l(Q_k)^s : E \subset \bigcup_{k = 1}^\infty Q_k \right\}. \]
The majority of Borel sets which occur in practice fail to be measurable with respect to $H^s_\infty$, but the exterior measure $H^s_\infty$ has the useful property that $H^s_\infty(E) = 0$ if and only if $H^s(E) = 0$.

\begin{lemma} \label{HausdorffBoundary}
	Consider $t < s$, and $E \subset \RR^d$.
	\begin{enumerate}
		\item[(i)] If $H^t(E) < \infty$, then $H^s(E) = 0$.
		\item[(ii)] If $H^s(E) \neq 0$, then $H^t(E) = \infty$.
	\end{enumerate}
\end{lemma}
\begin{proof}
	Suppose that $H^t(E) = A < \infty$. Then for any $\delta > 0$, there is a cover of $E$ by a collection of intervals $\{ Q_k \}$, such that $l(Q_k) \leq \delta$ for each $k$, and
	\[ \sum_{k = 1}^\infty l(Q_k)^t \leq A < \infty. \]
	But then
	\[ H^s_\delta(E) \leq \sum_{k = 1}^\infty l(Q_k)^s \leq \sum_{k = 1}^\infty l(Q_k)^{s-t} l(Q_k)^t \leq \delta^{s-t} A. \]
	As $\delta \to 0$, we conclude $H^s(E) = 0$, proving \emph{(i)}. And \emph{(ii)} is just the contrapositive of (i), and therefore immediately follows.
\end{proof}

\begin{corollary} \label{corollaryhausdorffzero}
	If $s > d$, $H^s = 0$.
\end{corollary}
\begin{proof}
	The measure $H^d$ is just the Lebesgue measure on $\RR^d$, so
	\[ H^d[-N,N]^d = (2N)^d. \]
	If $s > d$, Lemma \ref{HausdorffBoundary} implies $H^s[-N,N]^d = 0$. By countable additivity, taking $N \to \infty$ shows $H^s(\RR^d) = 0$. Since $H^s$ is a positive measure, $H^s(E) = 0$ for all $E$.
\end{proof}

Given any Borel set $E$, Corollary \ref{corollaryhausdorffzero}, combined with Lemma \ref{HausdorffBoundary}, implies there is a unique value $s_0 \in [0,d]$ such that $H^s(E) = 0$ for $s > s_0$, and $H^s(E) = \infty$ for $0 \leq s < s_0$. We refer to $s_0$ as the \emph{Hausdorff dimension} of $E$, denoted $\hausdim(E)$.

\begin{theorem}
	For any bounded set $E$, $\hausdim(E) \leq \lowminkdim(E) \leq \upminkdim(E)$.
\end{theorem}
\begin{proof}
	Given $l > 0$, we have a simple bound $H^s_l(E) \leq N(l,E) \cdot l^s$. If $\lowminkdim(E) < s$, then there exists a sequence $\{ l_k \}$ with $l_k \to 0$, and $N(l_k,E) \leq (1/l_k)^s$. We conclude that
	\[ H^s(E) = \lim_{k \to \infty} H^s_{l_k}(E) \leq \lim_{k \to \infty} N(l_k,E) \cdot l_k^s \leq 1. \]
	Thus Lemma \ref{HausdorffBoundary} implies $\hausdim(E) \leq s$. Taking infima over all $s > \lowminkdim(E)$ shows $\hausdim(E) \leq \lowminkdim(E)$.
\end{proof}

\begin{remark}
	If $\hausdim(E) < d$, then $|E| = H^d(E) = 0$. Thus any set with fractal dimension less than $d$ (either Hausdorff of Minkowski) must have measure zero. This means we can use the dimension as a way of distinguishing between sets of measure zero, which is precisely what we need to study the configuration avoidance problem for non-discrete configurations.
\end{remark}

The fact that Hausdorff dimension is defined over multiple scales simultaneously makes it more stable under analytical operations. In particular, for any family of at most countably many sets $\{ E_k \}$,
\[ \hausdim \left\{ \bigcup E_k \right\} = \sup \left\{ \hausdim(E_k) \right\}. \]
This need not be true for the Minkowski dimension; a single point has Minkowski dimension zero, but $\mathbf{Q} \cap [0,1]$, which is a countable union of points, has Minkowski dimension one. An easy way to make Minkowski dimension countably stable is to define the \emph{modified Minkowski dimensions}
\begin{align*}
	\lmbdim(E) &= \inf \left\{ s : E \subset \bigcup_{i = 1}^\infty E_i, \lowminkdim(E_i) \leq s\ \text{for each $i$} \right\}
\end{align*}
and
\begin{align*}
	\umbdim(E) &= \inf \left\{ s : E \subset \bigcup_{i = 1}^\infty E_i, \upminkdim(E_i) \leq s\ \text{for each $i$} \right\}.
\end{align*}
This notion of dimension, in a disguised form, appears in Chapter \ref{ch:RoughSets}.

\section{Dyadic Scales} \label{sec:Dyadics}

It is now useful to introduce the dyadic notation we utilize throughout this thesis. At the cost of some techniques which can be used by exploiting the full continuous structure of $\RR^d$, applying dyadic techniques often allows us to elegantly discretize problems in Euclidean space.

Fix an integer $N$. The classic family of dyadic cubes with \emph{branching factor} $N$ is given by setting, for each integer $k \geq 0$,
\[ \DD_k^d = \left\{ \prod_{i = 1}^d \left[ \frac{n_i}{N^k}, \frac{n_i + 1}{N^k} \right] : n \in \ZZ^d \right\}, \]
and then setting $\DD^d = \bigcup_{k \geq 0} \DD_k^d$. Elements of $\DD^d$ are known as \emph{dyadic cubes}, and elements of $\DD_k^d$ are known as \emph{dyadic cubes of generation $k$}. The most important properties of the dyadic cubes is that for each $k$, $\DD_k^d$ is a cover of $\RR^d$ by cubes of sidelength $1/N^k$, and for any two cubes $Q_1,Q_2 \in \DD^d$, either their interiors are disjoint, or one cube is nested in the other.
\begin{itemize}
	\item For each cube $Q_1 \in \DD_{k+1}^d$, there is a unique cube $Q_2 \in \DD_k^d$ such that $Q_1 \subset Q_2$. We refer to $Q_2$ as the \emph{parent} of $Q_1$, and $Q_1$ as a \emph{child} of $Q_2$. Each cube in $\DD$ has exactly $N^d$ children. For $Q \in \DD_{k+1}^d$, we let $Q^* \in \DD_k^d$ denote its parent.

	\item We say a set $E \subset \RR^d$ is \emph{$\DD_k$ discretized} if it is a union of cubes in $\DD_k^d$. If $E$ is $\DD_k$ discretized, we define
	\[ \DD_k(E) = \{ Q \in \DQ_k^d : Q \subset E \}. \]
	Then $E = \bigcup \DD_k(E)$.

	\item Given $k \geq 0$, and $E \subset \RR^d$, we let
	\[ E(1/N^k) = \bigcup \{ Q \in \DD_k^d : Q \cap E \neq \emptyset \}. \]
	Then $E(1/N^k)$ is the smallest $\DD_k$ discretized set containing $E$ in its interior. Our choice of notation invites thinking of $E(1/N^k)$ as a discretized version of the classic $1/N^k$ thickening
	\[ \{ x \in \RR^d : d(x,E) < 1/N^k \}. \]
	We have no need for the standard notion of thickening in this thesis, so there is no notational conflict.
\end{itemize}
Since any cube is covered by at most $O_d(1)$ cubes in $\DD$ of comparable sidelength, from the perspective of geometric measure theory, working with dyadic cubes is normally equivalent to working with the class of all cubes.

Our main purpose with working with dyadic cubes is to construct \emph{fractal-type sets}. By this, we mean defining sets $X$ as the intersection of a nested family of sets $\{ X_k \}$, where each $X_k$ is $\DD_k$ discretized, and each successive set $X_{k+1}$ is obtained from $X_k$ by application of a simple, recursive procedure. Such a construction satisfies the following three properties of Falconer's definition of a fractal, as detailed in the introduction to \cite{Falconer}:
\begin{itemize}
	\item[(i)] $X$ has detail at arbitrarily small scales.
	\item[(ii)] $X$ is too irregular to be described in traditional geometric language.
	\item[(v)] $X$ is defined recursively.
\end{itemize}
This justifies the term `fractal' when used to refer to these sets.

\begin{example}
	Let us construct the middle thirds Cantor set $C$ as a fractal-type set. We form $C$ from the family of dyadic cubes with branching factor $N = 3$. We initially set $C_0 = [0,1]$. Then, given the $\DD_k$ discretized set $C_k$, we consider each $I \in \DD_k^1(C_k)$, and let $\DD_{k+1}^1(I) = \{ I_1, I_2, I_3 \}$, where $I_1, I_2, I_3$ are given in increasing order with respect to their appearance in $I$. We set
	\[ C_{k+1} = \bigcup \{ I_1 \cup I_3 : I \in \DD_k^1(C_k) \}. \]
	Then $C = \bigcap_{k \geq 0} C_k$ is the Cantor set.
\end{example}

Unfortunately, a \emph{constant} branching factor is not sufficient to describe the fractal type constructions we discuss in this thesis. Thus, we introduce a more general family of cubes, which, abusing terminology, we also refer to as dyadic. Instead of a single branching factor $N$, we fix a sequence of positive integers $\{ N_k : k \geq 1 \}$, with $N_k \geq 2$ for all $k$, which gives the branching factor at each stage of the class of cubes we define.
\begin{itemize}
	\item For each $k \geq 0$, we define 
	\[ \DQ_k^d = \left\{ \prod_{i = 1}^d \left[ \frac{m_i}{N_1 \dots N_k}, \frac{m_i + 1}{N_1 \dots N_k} \right] : m \in \ZZ^d \right\}. \]
	These are the \emph{dyadic cubes of generation $k$}. We let $\DQ^d = \bigcup_{k \geq 0} \DQ_k^d$. Note that any two cubes $\DQ^d$ are either nested within one another, or their interiors are disjoint.

	\item We set $l_k = (N_1 \dots N_k)^{-1}$. Then $l_k$ is the sidelength of the cubes in $\DQ_k^d$.

	\item Given $Q \in \DQ_{k+1}^d$, we let $Q^* \in \DQ_k^d$ denote the \emph{parent cube} of $Q$, i.e. the unique dyadic cube of generation $k$ such that $Q \subset Q^*$.

	\item We say a set $E \subset \RR^d$ is \emph{$\DQ_k$ discretized} if it is a union of cubes in $\DQ_k^d$. In this case, we let
	\[ \DQ_k(E) = \{ Q \in \DQ_k^d: Q \subset E \} \]
	denote the family of cubes whose union is $E$.

	\item For $E \subset \RR^d$ and $k \geq 0$, we let $E(l_k) = \{ Q \in \DQ_k^d : Q \cap E = \emptyset \}$.
\end{itemize}
%
%The most common class of dyadic cubes in analysis is obtained by setting $N_k = 2$ for each $k$. We reserve a special notation for this class of dyadic cubes; the class of all such cubes is denoted by $\DD^d$, and the generation $k$ cubes by $\DD^d_k$. The fact that the sequence $\{ N_k \}$ is constant makes these cubes more easy to analyze. But in our methods, it is necessary for the sequence $\{ N_k \}$ to become unbounded as $k \to \infty$. This is why we have to introduce the more general family of dyadic cubes introduced above.
%
Sometimes, our recursive constructions need a family of `intermediary' cubes that lie between the scales $\DQ_k^d$ and $\DQ_{k+1}^d$. In this case, we consider a supplementary sequence $\{ M_k : k \geq 1 \}$ with $M_k \divides N_k$ for each $k$.
\begin{itemize}
	\item For $k \geq 1$, we define
	\[ \DR_k^d = \left\{ \prod_{i = 1}^d \left[ \frac{m_i}{N_1 \dots N_{k-1} M_k} , \frac{m_i + 1}{N_1 \dots N_{k-1} M_k} \right] : m \in \mathbf{Z}^d \right\}. \]

	\item We set $r_k = (N_1 \dots N_{k-1} M_k)^{-1}$. Then $r_k$ is the sidelength of a cube in $\DR_k^d$.

	\item The notions of being $\DR_k^d$ discretized, the collection of cubes $\DR_k^d(E)$, and the sets $E(r_k)$, are defined as should be expected.
\end{itemize}
The cubes in $\DR_k^d$ are coarser than those in $\DQ_k^d$, but finer than those in $\DQ_{k-1}^d$.

In this thesis, when we say we apply a \emph{single-scale construction}, we utilize the cubes $\DQ^d$, for an appropriate choice of parameters $\{ N_k \}$, as well as the notation given in this section. When we apply a \emph{multi-scale construction}, we utilize both the cubes in $\DQ^d$ and $\DR^d$, for an appropriate choice of parameters $\{ N_k \}$ and $\{ M_k \}$, and use the notation introduced above.

\begin{remark}
	We note that there is some notational conflict between the cubes $\DD^d$ and the cubes $\DQ^d$, but since we never use both families simultaneously in a single argument, it should be clear which notation we are using.
\end{remark}

\section{Frostman Measures}

%\begin{example}
%	Let $s = 0$. Then $H_\delta^0(E)$ is the number of $\delta$ balls it takes to cover $E$, which tends to $\infty$ as $\delta \to 0$ unless $E$ is finite, and in the finite case, $H_\delta^0(E) \to \# E$. Thus $H^0$ is just the counting measure.
%\end{example}

%\begin{example}
%	Let $s = d$. If $E$ has Lebesgue measure zero, then for any $\varepsilon > 0$, there exists a sequence of balls $\{ B(x_k,r_k) \}$ covering $E$ with
	%
%	\[ \sum_{k = 1}^\infty r_k^d < \varepsilon^d. \]
	%
%	Then we know $r_k < \varepsilon$, so $H^s_\varepsilon(E) < \varepsilon^d$. Letting $\varepsilon \to 0$, we conclude $H^d(E) = 0$. Thus $H^d$ is absolutely continuous with respect to the Lebesgue measure. The measure $H^d$ is translation invariant, so $H^d$ is actually a constant multiple of the Lebesgue measure.
%\end{example}

It is often easy to upper bound Hausdorff dimension, but non-trivial to \emph{lower bound} the Hausdorff dimension of a given set. A key technique to finding a lower bound is \emph{Frostman's lemma}, which says that a set has large Hausdorff dimension if and only if it supports a Borel measure obeying a decay law on small sets. We say a finite Borel measure $\mu$ is a \emph{Frostman measure} of dimension $s$ if it is non-zero, compactly supported, and there exists $C > 0$ such that for any cube $Q$, $\mu(Q) \leq C \cdot l(Q)^s$. The proof of Frostman's lemma will utilize a technique often useful, known as the \emph{mass distribution principle}. To prove the mass distribution principle, we apply weak convergence.

\begin{lemma} \label{weakstarcompleteness}
	Suppose $\{ \mu_i \}$ is a Cauchy sequence of non-negative, regular Borel measures on $\RR^d$, in the sense that for any $f \in C_c(\RR^d)$, the sequence
	\[ \left\{ \int f d\mu_i \right\} \]
	is Cauchy. Then there is a regular Borel measure $\mu$ such that $\mu_i \to \mu$ vaguely, in the sense that for any $f \in C_C(\RR^d)$,
	\[ \int f d\mu = \lim_{i \to \infty} \int f d\mu_i. \]
\end{lemma}
\begin{proof}
	Fix a compact set $K$. Then we can find a function $\varphi \in C_c(\RR^d)$ such that $\mathbf{I}_K \leq \varphi$. This means that
	\[ \mu_i(K) = \int \mathbf{I}_K d\mu_i \leq \int \varphi d\mu_i. \]
	Since $\{ \int \varphi d\mu_i \}$ is Cauchy, the values $\int \varphi d\mu_i$ are uniformly bounded in $i$. In particular, $\mu_i(K)$ is uniformly bounded in $i$. Applying the Banach Alaoglu theorem, we conclude that the collection of measures $\{ \mu_i|_K \}$ is contained in a compact subset of the space of non-negative Borel measures with respect to the vague topology. But every compact subset of a locally convex space is complete, and so we can therefore find a finite Borel measure $\mu^K$ supported on $K$ such that $\mu_i|_K \to \mu^K$ vaguely.

	If $f \in C_c(\RR^d)$ is supported on $K_1 \cap K_2$, for two compact sets $K_1$ and $K_2$, then
	\[ \int f d\mu^{K_1} = \lim_{i \to \infty} \int_{K_1} f d\mu_i = \lim_{i \to \infty} \int_{K_2} f d \mu_i = \int f d\mu^{K_2}. \]
	Thus we can define a measure $\mu$ such that for each $f \in C_c(\RR^d)$, if $f$ is supported on a compact set $K$, then
	\[ \int f d\mu = \int f d\mu^K. \]
	For any $f \in C_c(\RR^d)$, $f$ is supported on some compact set $K$, and then
	\[ \int f d\mu = \int f d\mu^K = \lim_{i \to \infty} \int_K f d\mu_i = \lim_{i \to \infty} \int f d\mu_i. \]
	Since $f$ was arbitrary, $\mu_i \to \mu$ in the vague topology.
\end{proof}

\begin{comment}

\begin{lemma}
	For each $k$, let $\mathcal{E}_k$ be a finite collection of closed sets, such that for any $A \in \mathcal{E}_{k+1}$, we fix $A^* \in \mathcal{E}_k$ such that $A \subset A^*$. Suppose that for any $A' \in \mathcal{E}_k$, there is at least one $A \in \mathcal{E}_{k+1}$ such that $A^* = A'$. Let $f: \bigcup_k \mathcal{E}_k \to [0,\infty)$ be a function such that for any $A' \in \mathcal{E}_k$,
	%
	\begin{equation} \label{equation73234091} \sum \{ f(A) : A^* = A' \} = f(A'). \end{equation}
	%
	Then there exists a regular Borel measure $\mu$ supported on
	%
	\[ \bigcap_{k = 1}^\infty \bigg\{ \overline{\bigcup \{ A : A \in \mathcal{E}_k : f(A) > 0 \}} \bigg\} \]
	%
	such that for each $k$, and each $A \in \mathcal{E}_k$,
	%
	\begin{equation} \label{massdissupperbound} \mu(A) \geq f(A), \end{equation}
	%
	and for any Borel $E$,
	%
	\begin{equation} \label{massdisslowerbound} \mu(E) \leq \sum f(A), \end{equation}
	%
	where $A$ ranges over elements of $\mathcal{E}_k$ with $A \cap E \neq \emptyset$.
\end{lemma}

\end{comment}

\begin{theorem}[Mass Distribution Principle] \label{massdistributionprinciplelem}
	Let $w: \DQ^d \to [0,\infty)$ be a function such that for any $Q_0 \in \DQ^d$,
	\begin{equation} \label{equation73234091} \sum_{Q^* = Q_0} w(Q) = w(Q_0). \end{equation}
	Then there exists a regular Borel measure $\mu$ supported on
	\[ \bigcap_{k = 1}^\infty \left[ \bigcup \{ Q \in \DQ_k^d : w(Q) > 0 \} \right], \]
	such that for each $Q \in \DQ^d$,
	\begin{equation} \label{massdissupperbound} \mu(Q^\circ) \leq w(Q) \leq \mu(Q), \end{equation}
	and for any set $E$ and $k \geq 0$, if $\mathcal{E}_k = \DQ_k(E(l_k))$, then
	\begin{equation} \label{massdisslowerbound} \mu(E) \leq \sum_{Q \in \mathcal{E}_k} w(Q). \end{equation}
\end{theorem}
\begin{proof}
	For each $i$, let $\mu_i$ be a regular Borel measure such that for each $j \leq i$, and $Q \in \DQ_j^d$, $\mu_i(Q) = \mu_i(Q^\circ) = w(Q)$. One such choice is given, for each $\varphi \in C_c(\RR^d)$, by the equation
	\[ \int \varphi d\mu_i = \sum_{Q \in \DQ_i^d} \frac{w(Q)}{|Q|} \int_Q \varphi\; dx. \]
	We claim that $\{ \mu_i \}$ is a Cauchy sequence. Fix $\varphi \in C_c(\RR^d)$, and choose some $N$ such that $\varphi$ is supported on $[-N,N]^d$. Set $Q' = [-N,N]^d$, and for each $i$, set $\mathcal{Q}'_i = \DQ_i(Q')$. Since $\varphi$ is compactly supported, $\varphi$ is \emph{uniformly continuous}, so for each $\varepsilon > 0$, if $i$ is suitably large, there is a sequence of values $\{ a_Q : Q \in \DQ_i^d(Q') \}$ such that if $x \in Q$, $|\varphi(x) - a_Q| \leq \varepsilon$. But this means
	\[ \sum_{Q \in \mathcal{Q}'_i} (a_Q - \varepsilon) \mathbf{I}_Q \leq \varphi \leq \sum_{Q \in \mathcal{Q}'_i} (a_Q + \varepsilon) \mathbf{I}_Q. \]
	Thus for any $j \geq i$,
	\[ \int \varphi\; d\mu_j\leq \sum_{Q \in \mathcal{Q}'_i} (a_Q + \varepsilon) \mu_j(Q) = \sum_{Q \in \mathcal{Q}'_i} (a_Q + \varepsilon) w(Q) \]
	and
	\[ \int \varphi\; d\mu_j \geq \sum_{Q \in \mathcal{Q}'_i} (a_Q - \varepsilon) \mu_j(Q) = \sum_{Q \in \mathcal{Q}'_i} (a_Q - \varepsilon) w(Q). \]
	In particular, if $j,j' \geq i$,
	\[ \left| \int \varphi d\mu_j - \int \varphi d\mu_{j'} \right| \leq 2 \varepsilon \sum_{Q \in \mathcal{Q}'_i} w(Q). \]
	Repeated applications of \eqref{equation73234091} show that
	\[ \sum_{Q \in \mathcal{Q}'_i} w(Q) = \sum_{Q \in \mathcal{Q}'_0} w(Q), \]
	which is therefore bounded independently of $i$. Since $\varepsilon$ and $\varphi$ were arbitrary, this shows $\{ \mu_i \}$ is Cauchy.

	Applying Lemma \ref{weakstarcompleteness}, we find a regular Borel measure $\mu$ such that $\mu_i \to \mu$ vaguely. For any $Q \in \DQ^d$, since $Q$ is a closed set,
	\[ \mu(Q) \geq \limsup_{i \to \infty} \mu_i(Q) = w(Q), \]
	and since $Q^\circ$ is an open set,
	\[ \mu(Q^\circ) \leq \liminf_{i \to \infty} \mu_i(Q^\circ) = w(Q). \]
	This establishes \eqref{massdissupperbound}. Conversely, if $E$ is arbitrary, then $E \subset E(l_i)^\circ$ for any $i$, so if $\mathcal{E}_i = \DQ_i(E(l_i))$, then
	\[ \mu(E) \leq \liminf_{i \to \infty} \mu_i(E(l_i)^\circ) \leq \liminf_{i \to \infty} \mu_i(E(l_i)) = \sum_{Q \in \mathcal{E}_i} w(Q). \]
	This establishes \eqref{massdisslowerbound}.
\end{proof}

\begin{remark}
	The reason why we cannot necessarily find a function $\mu$ which \emph{precisely} extends $w$ is that in the weak limit, mass which the weight function $w$ assigns to one cube can `leak' into the mass of adjacent cubes when we take a weak limit. This is why we must use the weaker `extension bounds' \eqref{massdissupperbound} and \eqref{massdisslowerbound}. This doesn't cause us to `gain' or `lose' any mass in the weak limit, as \eqref{massdissupperbound} shows, so the result is still a `mass distribution' result. It just means that the mass specified by $w$ may be shared by adjacent cubes in the weak limit, if enough mass is pushed out to the boundary of these cubes. If, for each $Q' \in \DQ^d$,
	\[ \lim_{k \to \infty} \sum_{Q \in \mathcal{Q}'_k} w(Q) = 0, \]
	where $\mathcal{Q}'_k = \DQ_k(Q'(l_k)) - \DQ_k(Q')$, then \eqref{massdissupperbound} and \eqref{massdisslowerbound} together imply we actually have $\mu(Q) = w(Q)$ for all $Q \in \DQ$. Thus $\mu$ is a measure extending $w$. One such condition that guarantees this is that $w(Q) \lesssim l(Q)^s$ for some $s > d-1$. Another is that for any $k \geq 0$, and distinct $Q_0, Q_1 \in \DQ_k^d$ with $Q_0 \cap Q_1 \neq \emptyset$, either $w(Q_0) = 0$ or $w(Q_1) = 0$.
\end{remark}

%\begin{lemma}
%	let $\mu^+$ be a function from $\B$ to $[0,\infty)$ such that for any $I \in \B(1/M^k,\RR^d)$,
	%
%	\[ \sum \left\{ \mu^+(J) :J \in \B(1/M^{k+1},I) \right\} \leq \mu^+(I) \]
	%
%	Assume there exists $c > 0$ such that for all $k$,
	%
%	\[ \sum \left\{ \mu^+(I) : I \in \B(1/M^k,I) \right\} \geq c \]
	%
%	and
	%
%	\[ \sum \left\{ \mu^+(I) : I \in \B(1,I) \right\} < \infty \]	
	%
%	Then there exists a non-zero Borel measure $\mu$ such that $\mu(I) \leq \mu^+(I)$ for $I \in \B$.
%\end{lemma}
%\begin{proof}
%	As in the last lemma, define the operators $E_k$ and the measures $\mu_k$. By weak compactness, a subsequence of these measures converge weakly to some measure $\mu$, and $E_k(\mu) = \lim E_k(\mu_{j_k}) \leq \mu_k$. The measure $\mu$ is nonzero, since $\| \mu_{j_k} \| \geq c$ for each $k$, and so $\| \mu \| \geq c$.
%\end{proof}

\begin{theorem}[Frostman's Lemma]
	If $E$ is a Borel set, $H^s(E) > 0$ if and only if there exists an $s$ dimensional Frostman measure supported on $E$.
\end{theorem}
\begin{proof}
	Suppose that $\mu$ is $s$ dimensional and supported on $E$. If $H^s(F) = 0$, then for each $\varepsilon > 0$ there is a sequence of cubes $\{ Q_k \}$ whose union covers $F$, with $\sum_{k = 1}^\infty l(Q_k)^s \leq \varepsilon$. But then
	\[ \mu(F) \leq \sum_{k = 1}^\infty \mu(Q_k) \lesssim \sum_{k = 1}^\infty l(Q_k)^s \leq \varepsilon. \]
	Taking $\varepsilon \to 0$, we conclude $\mu(F) = 0$. Thus $\mu$ is absolutely continuous with respect to $H^s$. Since $\mu(E) > 0$, this means that $H^s(E) > 0$.

	To prove the converse, we will suppose for simplicity that $E$ is compact. We work dyadically with the classical family of dyadic cubes $\DD^d$, with branching factor $N = 2$.  By translating, we may assume that $H^s(E \cap [0,1]^d) > 0$, and so without loss of generality we may assume $E \subset [0,1]^d$. For each $Q \in \DD^d$, define $w^+(Q) = H^s_\infty(E \cap Q)$. Then
	\begin{equation} \label{equation7099014901209} w^+(Q) \leq l(Q)^s, \end{equation}
	and $w^+$ is subadditive. We now recursively define a function $w$ such that for any $Q \in \DQ^d$, \eqref{equation73234091} and
	\begin{equation} \label{equation104-240-640954} w(Q) \leq w^+(Q), \end{equation}
	are satisfied at each stage of the definition of $w$. We initially define $w$ by setting $w([0,1]^d) = w^+([0,1]^d)$. Given $Q \in \DD_k^d$, we enumerate its children as $Q_1, \dots, Q_M \in \DD_{k+1}^d$. We then consider any values $A_1, \dots, A_M \geq 0$ such that
  	\begin{equation} \label{equation3424209034}
  		A_1 + \dots + A_M = w(Q),
  	\end{equation}
  	and for each $k$,
  	\begin{equation} \label{equation12039123012}
  		A_k \leq w(Q_k).
  	\end{equation}
	This is feasible to do because $w^+(Q_1) + \dots + w^+(Q_M) \geq w^+(Q)$, and \eqref{equation104-240-640954} holds for the previous definition of $w$. We then define $w(Q_k) = A_k$ for each $k$. Equation \eqref{equation3424209034} implies \eqref{equation73234091} holds for this new set of definitions, and \eqref{equation12039123012} implies \eqref{equation104-240-640954} holds. Thus $w$ is a well defined function. Furthermore, \eqref{equation3424209034} implies \eqref{equation73234091} of Lemma \ref{massdistributionprinciplelem}, and so the mass distribution principle gives the existence of a measure $\mu$ supported on $E$, satisfying \eqref{massdisslowerbound} and \eqref{massdissupperbound}. In particular, \eqref{massdisslowerbound} implies $\mu$ is non-zero. For each $Q' \in \DD_k^d$, $\#[\DD_k^d(Q(l_k))] = 3^d = O_d(1)$, so if we define $\mathcal{Q}' = \DD_k^d(Q(l_k))$, then \eqref{massdissupperbound}, \eqref{equation7099014901209}, and \eqref{equation104-240-640954} imply
	\[ \mu(Q') \lesssim_d \max_{Q \in \mathcal{Q}'} w^+(Q) \leq 1/2^{ks} = l(Q)^s. \]
	where $Q'$ ranges over the cubes in $\DD_k^d(Q(l_k))$. Given any cube $Q$, we find $k$ with $1/2^{k-1} \leq l(Q) \leq 1/2^k$. Then $Q$ is covered by $O_d(1)$ dyadic cubes in $\DD_k^d$, and so $\mu(Q) \lesssim l(Q)^s$. Thus $\mu$ is a Frostman measure of dimension $s$.
\end{proof}

Given any finite Borel measure $\mu$, we let
\[ \hausdim(\mu) = \left\{ s : \mu\ \text{is a Frostman measure of dimension $s$} \right\}. \]
Frostman's lemma says that for any Borel set $E$, $\hausdim(E)$ is the supremum of $\hausdim(\mu)$, over all measures $\mu$ supported on a closed subset of $E$. We refer to $\hausdim(\mu)$ as the \emph{Frostman dimension} of the measure $\mu$.

\section{Fourier Dimension}

A popular technique in current research in geometric measure theory is exploiting Fourier analysis to obtain additional structural information about configurations in sets. A key insight to this technique is that the Frostman dimension of any finite Borel measure $\mu$ is equal to
\[ \sup \left\{ s > 0 : \int \frac{|\widehat{\mu}(\xi)|^2}{|\xi|^{d-s}} d\xi < \infty \right\}. \]
For brevity, we leave the proof to other sources, e.g. \cite[Section 3.5]{Matilla}. Note that if
\[ \int \frac{|\widehat{\mu}(\xi)|^2}{|\xi|^{d-s}}\; d\xi < \infty, \]
then there exists a constant $C$ such that for \emph{most} values $\xi \in \RR^d$,
\begin{equation} \label{fourierdimensioncondition}
    |\widehat{\mu}(\xi)| \leq C |\xi|^{-s/2}.
\end{equation}
We obtain a strengthening of the Frostman measure condition if we require \eqref{fourierdimensioncondition} to hold for \emph{all} values $\xi$. In particular, we define the \emph{Fourier dimension} of a finite Borel measure $\mu$ on $\RR^d$ as
\[ \fordim(\mu) = \sup \{ 0 < s \leq d : \sup_{\xi \in \RR^d} |\xi|^s |\widehat{\mu}(\xi)| \}. \]
If this is true, then for all $t < s$,
\[ \int \frac{|\widehat{\mu}(\xi)|^2}{|\xi|^{d-t}}\; d\xi < \infty \]
so $\mu$ has Frostman dimension $t$ for all $t < s$. Thus if we define the \emph{Fourier dimension} of a set $E$ as
\[ \fordim(E) = \sup \left\{ \fordim(\mu): \mu\ \text{is supported on $E$} \right\}, \]
then $\fordim(E) \leq \hausdim(E)$.

We view the Fourier dimension as a refinement of the Hausdorff dimension which gives greater structural control on the set in the `frequency domain'. Most classical examples of fractals, like the middle-thirds Cantor set, have Fourier dimension zero. Nonetheless, one principle in this setting is that the Fourier dimension of \emph{random} families of sets tend to almost surely have Fourier dimension equal to their Hausdorff dimension. Our main technique to construct configuration avoiding sets in Chapter \ref{ch:RoughSets} involves a random selection strategy, and so in Section 6.2, we attempt to utilize this random selection strategy to find sets with large Fourier dimension avoiding configurations.

\section{Dyadic Fractal Dimension}

It is often natural for us to establish results about fractal dimension `dyadically', working with the family of cubes $\DQ^d$ and branching factors $\{ N_k : k \geq 1 \}$. We begin with Minkowski dimension. For each $m$, let $N_{\DQ}(m,E)$ denote the minimal number of cubes in $\DQ_m^d$ required to cover $E$. This is often easy to calculate, up to a multiplicative constant, by greedily selecting cubes which intersect $E$.

\begin{lemma} \label{comparableCovers}
	For any set $E$,
	\[ N_{\DQ}(m,E) \sim_d \# \{ Q \in \DQ_m^d : Q \cap E \neq \emptyset \} \sim_d N(l_m,E). \]
\end{lemma}
\begin{proof}
	Let $\mathcal{E} = \{ Q \in \DQ_m^d : Q \cap E \neq \emptyset \}$. Then $N(l_m,E) \leq N_{\DQ}(m,E) \leq \#(E)$. Conversely, let $\{ Q_k \}$ be a minimal cover of $E$ by cubes. Then each cube $Q_k$ intersects at most $3^d$ cubes in $\DQ_k^d$, so $\#(\mathcal{E}) \leq 3^d \cdot N(l_m,E) \leq 3^d \cdot N_{\DQ}(m,E)$.
\end{proof}

Thus it is natural to ask whether it is true that for any set $E$,
\begin{equation} \label{definingSequence}
	\begin{aligned}
		\lowminkdim(E) &= \liminf_{k \to \infty} \frac{\log[N_{\DQ}(k,E)]}{\log[1/l_k]}\\
		\text{and}\\
		\upminkdim(E) &= \limsup_{k \to \infty} \frac{\log[N_{\DQ}(k,E)]}{\log[1/l_k]}.
	\end{aligned}
\end{equation}
The answer depends on the choice of $\{ N_k \}$. In particular, a sufficient condition (and as we see later, essentially necessary) is that
\begin{equation} \label{definingsequencegrowthrate}
	N_{k+1} \lesssim_\varepsilon (N_1 \dots N_k)^\varepsilon \quad \text{for any $\varepsilon > 0$}.
\end{equation}
We will see that this condition allows us to work dyadically in many scenarios when it comes to fractal dimension.

%Subsets of $\Sigma^d$ can also be assigned a Minkowski dimension. We define
%
%\[ \lowminkdim(E) = \liminf_{k \to \infty} \frac{\log(\#(\sigma_k^d(E)))}{\log(1/l_k)}\quad\text{and}\quad\upminkdim(E) = \limsup_{k \to \infty} \frac{\log(\#(\sigma_k^d(E)))}{\log(1/l_k)}. \]
%
%This makes sense, because $\Sigma^d$ only really has `balls' of radius $\{ l_k \}$, for each $k$, and \emph{any} cover of $E$ by balls of radius $l_k$ contains $\sigma_k^d(E)$. In order 
%we have $N(E,l_k) = \# (\Sigma_k^d(E))$, since \emph{any} cover of $E$ by balls of radius $l_k$

\begin{theorem} \label{definingsequenceminkowski}
	If \eqref{definingsequencegrowthrate} holds, then \eqref{definingSequence} holds.
\end{theorem}
\begin{proof}
	Fix a length $l$, and find $k$ with $l_{k+1} \leq l \leq l_k$. Applying Lemma \ref{comparableCovers} shows
	\[ N(l,E) \leq N(l_{k+1},E) \lesssim_d N_{\DQ}(k+1,E) \]
	and
	\[ N(l,E) \geq N(l_k,E) \gtrsim_d N_{\DQ}(k,E). \]
	Thus
	\[ \frac{\log[N(l,E)]}{\log[1/l]} \leq \left[ \frac{\log(1/l_{k+1})}{\log(1/l_k)} \right] \frac{\log[N_{\DQ}(k+1,E)]}{\log[1/l_{k+1}]} + O_d(1/k) \]
	and
	\[ \frac{\log[N(l,E)]}{\log[1/l]} \geq \left[ \frac{\log(1/l_k)}{\log(1/l_{k+1})} \right] \frac{\log[N_{\DQ}(k,E)]}{\log[1/l_k]} + O_d(1/k). \]
	Provided that
	\begin{equation} \label{equivalenceofscales}
		\frac{\log(1/l_{k+1})}{\log(1/l_k)} \to 1,
	\end{equation}
	the conclusion of the theorem is true. But \eqref{equivalenceofscales} is equivalent to the condition that
	\[ \frac{\log(N_{k+1})}{\log(N_1) + \dots + \log(N_k)} \to 0, \]
	and this is equivalent to \eqref{definingsequencegrowthrate}.
\end{proof}

Any constant branching factor satisfies \eqref{definingsequencegrowthrate} for the Minkowski dimension. In particular, we can work fairly freely with the classical dyadic cubes without any problems occurring. But more importantly for our work, we can let the sequence $\{ N_k \}$ increase rapidly.

\begin{theorem} \label{rapidBranching}
	If $N_k = 2^{\lfloor 2^{k \psi(k)} \rfloor}$, where $\psi(k)$ is any decreasing sequence of positive numbers tending to zero, such that
	\begin{equation} \label{definingsequencegrowthassumption} \psi(k) \geq \log_2(k)/k, \end{equation}
	then \eqref{definingsequencegrowthrate} holds.
\end{theorem}
\begin{proof}
	We note that $\log(N_k) = 2^{k \psi(k)} + O(1)$, and that \eqref{definingsequencegrowthassumption} implies that
	\[ 2^{\psi(1)} + \dots + 2^{k \psi(k)} \geq k. \]
	Putting these two facts together, we conclude that
	\begin{align*}
		\frac{\log(N_{k+1})}{\log(N_1) + \dots + \log(N_k)} &= \frac{2^{(k+1) \psi(k+1)} + O(1)}{2^{\psi(1)} + 2^{2 \psi(2)} + \dots + 2^{k \psi(k)} + O(k)}\\
		&\lesssim \frac{2^{(k+1) \psi(k+1)}}{2^{\psi(k)} + 2^{2 \psi(k)} + \dots + 2^{k \psi(k)}}\\
		&\lesssim \frac{2^{(k+1) \psi(k+1)}}{2^{(k+1) \psi(k)}} ( 2^{\psi(k)} - 1 )\\
		&\leq (2^{\psi(k)} - 1) \to 0.
	\end{align*}
	This is equivalent to \eqref{definingsequencegrowthrate}.
\end{proof}
We refer to any sequence $\{ l_k \}$ constructed by $\{ N_k \}$ satisfying \eqref{definingsequencegrowthassumption} for some function $\psi$ as a \emph{subhyperdyadic} sequence. If a sequence $\{ l_k \}$ is generated by a sequence $\{ N_k \}$ such that for some fixed $c > 0$,
\[ N_k = 2^{\lfloor 2^{ck} \rfloor}, \]
then the values $\{ l_k \}$ are referred to as \emph{hyperdyadic}. The next (counter) example shows that hyperdyadic sequences are essentially the `boundary' for sequences that can be used to measure the Minkowski dimension.

\begin{example}
	We consider a multi-scale dyadic construction, utilizing the two families $\DQ^d$ and $\DR^d$. Fix $0 \leq c < 1$, and define $N_k = 2^{\lfloor 2^{ck} \rfloor}$, and $M_k = 2^{\lfloor c 2^{ck} \rfloor}$. Then $M_k \divides N_k$ for each $k$. We recursively define a nested family of sets $\{ E_k \}$, with each $E_k$ a $\DQ_k^d$ discretized set, and set $E = \bigcap E_k$. We define $E_0 = [0,1]$. Then, given $E_k$, for each $Q \in \DQ_k(E_k)$, we select a \emph{single} cube $R_Q \in \DR_{k+1}(E_k)$, and define $E_{k+1} = \bigcup R_Q$. Then $\#(\DQ_0(E_0)) = 1$, and
	\[ \#(\DQ_{k+1}(E_{k+1})) = (N_{k+1}/M_{k+1}) \#(\DQ_k(E_k)), \]
	which we can simplify to read
	\begin{equation} \label{equation12623} \#(\DQ_k(E_k)) = \frac{N_1 \dots N_k}{M_1 \dots M_k}. \end{equation}
	Noting that $\log(N_i) = 2^{ci} + O(1)$, and $\log(M_i) = c2^{ci} + O(1)$, we conclude that
	\begin{align*}
		\frac{\log \#(\DQ_k(E_k))}{\log(1/l_k)} &= \frac{(1-c)(2^c + \dots + 2^{ck}) + O(k)}{(2^c + \dots + 2^{ck}) + O(k)} \to 1-c.
	\end{align*}
	On the other hand, for each $k$,
	\[ \#(\DR_{k+1}^d(E_k)) = \#(\DQ_k(E_k)) = \frac{N_1 \dots N_k}{M_1 \dots M_k}, \]
	and so
	\begin{align*}
		\frac{\log \#(\DR_{k+1}^d(E_k))}{\log(1/r_{k+1})} &= \frac{(1-c)(2^c + \dots + 2^{ck}) + O(k)}{(2^c + \dots + 2^{ck}) + c2^{c(k+1)} + O(k)} \\
		&= \frac{(1-c) \cdot 2^{c(k+1)} + O(k)}{(1 - c + c 2^c) \cdot 2^{c(k+1)} + O(k)}\\
		&\to \frac{1 - c}{1 - c + c2^c} < 1 - c.
	\end{align*}
	In particular,
	\[ \lowminkdim(E) \neq \liminf_{k \to \infty} \frac{\log \left[ N(l_k,E) \right]}{\log(1/l_k)}, \] % = \liminf_{k \to \infty} \frac{\log[N_{\DQ}(k,E)]}{\log(1/l_k)}, \]
	so measurements at hyperdyadic scales fail to establish general results about the Minkowski dimension.

\end{example}

%In particular, we can define the Minkowski dimensions of $E \subset \Sigma$ as
%
%\[ \lowminkdim(E) = \liminf_{k \to \infty} \frac{\log(\#(\Sigma_k^d(E)))}{\log(1/l_k)}\quad\text{and}\quad\upminkdim(E) = \limsup_{k \to \infty} \frac{\log(\#(\Sigma_k^d(E)))}{\log(1/l_k)}. \]
%
%Similarily, the Hausdorff measures $H^s$ are obtained by setting
%
% TODO: Fix this
%\[ H^s(E) = \left\{ \sum_m l_{k_m} : Q_m \in \Sigma_{k_m}^d\ \text{for each $k$}, \text{For any} \right\} \]
%
%and define the Hausdorff dimension correspondingly. A natural question is whether $\dim(\pi(E)) = \dim(E)$ for the various fractal dimensions we consider in this thesis. This is addressed in the next section.

We now move on to calculating Hausdorff dimension dyadically. The natural quantity to consider is the measure defined for any set $E$ as
\[ H^s_{\DQ}(E) = \lim_{m \to \infty} H^s_{\DQ,m}(E), \]
where
\[ H^s_{\DQ,m}(E) = \inf \left\{ \sum_k l(Q_k)^s : E \subset \bigcup_k^\infty Q_k,\ Q_k \in \bigcup_{i \geq m} \DQ_i^d\ \text{for each $k$} \right\}. \]
A similar argument to the standard Hausdorff measures shows there is a unique $s_0$ such that $H^s_{\DQ}(E) = \infty$ for $s < s_0$, and $H^s_{\DQ}(E) = 0$ for $s > s_0$. It is obvious that $H^s_{\DQ}(E) \geq H^s(E)$ for any set $E$, so we certainly have $s_0 \geq \hausdim(E)$. The next lemma guarantees that $s_0 = \hausdim(E)$, under the same conditions on the sequence $\{ N_k \}$ as found in Theorem \ref{definingsequenceminkowski}.

\begin{lemma} \label{lemma51464}
	If \eqref{definingsequencegrowthrate} holds, then for any $\varepsilon > 0$, $H^s_{\DQ}(E) \lesssim_{s,\varepsilon} H^{s-\varepsilon}(E)$.
\end{lemma}
\begin{proof}
	Fix $\varepsilon > 0$ and $m$. Let $E \subset \bigcup Q_k$, where $l(Q_k) \leq l_m$ for each $k$. Then for each $k$, we can find $i_k$ such that $l_{i_k+1} \leq l(Q_k) \leq l_{i_k}$. Then $Q_k$ is covered by $O_d(1)$ elements of $\DQ_{i_k}^d$, and
	\begin{equation} \label{equation824} H^s_{\DQ,m}(E) \lesssim_d \sum l_{i_k}^s \leq \sum (l_{i_k}/l_{i_k+1})^s l(Q_k)^s \leq \sum \left( l_{i_k}/l_{i_{k+1}} \right)^s l_{i_k}^\varepsilon l(Q_k)^{s - \varepsilon}. \end{equation}
	By assumption,
	\begin{equation} \label{equation992352}
		l_{i_k+1} = \frac{1}{N_1 \dots N_{i_k} N_{i_k+1}} \gtrsim_{s,\varepsilon} (N_1 \dots N_{i_k})^{1+ \varepsilon/s} = l_{i_k}^{1 + \varepsilon/s}.
	\end{equation}
	Putting \eqref{equation824} and \eqref{equation992352} together, we conclude that $H^s_{\DQ,m}(E) \lesssim_{d,s,\varepsilon} \sum l(Q_k)^{s-\varepsilon}$. Since $\{ Q_k \}$ was an arbitrary cover of $E$, we conclude $H^s_{\DQ,m}(E) \lesssim H^{s-\varepsilon}(E)$, and since $m$ was arbitrary, that $H^s_{\DQ}(E) \lesssim H^{s-\varepsilon}(E)$.
\end{proof}

Finally, we consider computing whether we can establish that a measure is a Frostman measure dyadically.

\begin{theorem} \label{easyCoverTheorem}
	If \eqref{definingsequencegrowthrate} holds, and if $\mu$ is a Borel measure such that $\mu(Q) \lesssim l(Q)^s$ for each $Q \in \DQ_k^d$, then $\mu$ is a Frostman measure of dimension $s - \varepsilon$ for each $\varepsilon > 0$.
\end{theorem}
\begin{proof}
	Given a cube $Q$, find $k$ such that $l_{k+1} \leq l(Q) \leq l_k$. Then $Q$ is covered by $O_d(1)$ cubes in $\DQ^d_k$, which shows
	\[ \mu(Q) \lesssim_d l_k^s = [(l_k/l)^s l^\varepsilon ] l^{s - \varepsilon} \leq [l_k^{s + \varepsilon} / l_{k+1}^s] l^{s - \varepsilon} = \left[ \frac{N_{k+1}^s}{(N_1 \dots N_k)^\varepsilon} \right] l^{s-\varepsilon} \lesssim_\varepsilon l^{s-\varepsilon}. \qedhere \]
\end{proof}	

Let's recognize the utility of this approach from the perspective of a dyadic construction. Suppose we have a sequence of nested sets $\{ E_k \}$, where $E_k$ is a $\DQ_k$ discretized subset of $[0,1]^d$, and for each $Q_0 \in \DQ_k(E_k)$, there is at least one cube $Q \in \DQ_{k+1}(E_{k+1})$ with $Q^* = Q_0$. Then we can set $E = \bigcap E_k$ as a `limit' of the discretizations $E_k$. We can associate with this construction a finite measure $\mu$ supported on $E$. It is defined by the mass distribution principle with respect to a function $w: \DQ^d \to [0,\infty)$. We set $w([0,1]^d) = 1$, and for each $Q \in \DQ_{k+1}(E_{k+1})$, set
\[ w(Q) = \frac{w(Q^*)}{\# \{ Q' \in \DQ_{k+1}(E_{k+1}) : (Q')^* = Q^* \}}. \]
The mass distribution principle then gives a Borel measure $\mu$, which we refer to as the \emph{canonical measure} associated with this construction. For this measure, it is often easy to show from a combinatorial argument that $w(Q) \lesssim l(Q)^s$ if $Q \in \DQ^d$, which together with \eqref{massdissupperbound} also implies $\mu(Q) \lesssim l(Q)^s$ for $Q \in \DQ^d$. This makes Theorem \ref{easyCoverTheorem} useful.

\begin{remark}
	The dyadic construction showing that Minkowski dimension cannot be measured only at hyperdyadic scales also shows that a bound on a measure $\mu$ on hyperdyadic cubes does not imply the correct bound at all scales. It is easy to show from \eqref{equation12623} that the canonical measure $\mu$ for this example satisfies, for each $k$, each $Q \in \DQ_k(E_k)$, and each $\varepsilon > 0$,
	\[ \mu(Q) = \left( \#(\DQ_k(E_k)) \right)^{-1} \lesssim_\varepsilon l(Q)^{1-c-\varepsilon} \]
	for all $Q \in \DQ_k(E_k)$, yet we know that for the set constructed in that example,
	\[ \hausdim(E) \leq \lowminkdim(E) < 1 - c. \]
	Frostman's lemma implies we cannot possibly have $\mu(Q) \lesssim_\varepsilon l(Q)^{1-c-\varepsilon}$ for all $\varepsilon > 0$ and \emph{all} cubes $Q$.
\end{remark}

%Our final method for interpolating requires extra knowledge of the dissection process, but enables us to choose the $l_k$ arbitrarily rapidly. The idea behind this is that there is an additional sequence of lengths $r_k$ with $l_k \leq r_k \leq l_{k-1}$. The difference between $r_k$ and $l_{k-1}$ is allowed to be arbitrary, but the decay rate between $l_k$ and $r_k$ is of polynomial-type, which enables us to use the covering methods of the previous section. In addition, we rely on a `uniform mass bound' between $r_k$ and $l_k$ to cover the remaining classes of intervals. Because we can take $r_k$ arbitrarily large relative to $l_k$, this renders any constants that occur in the construction to become immediately negligible. For two quantities $A$ and $B$, we will let $A \precsim_k B$ stand for an inequality with a hidden constant depending only on parameters with index smaller than $k$, i.e. $A \leq C(l_1, \dots, l_k, r_1,\dots,r_k) B$ for some constant $C(l_1, \dots, l_k, r_1, \dots, r_k)$ depending only on parameters with indices up to $k$.

\section{Beyond Hyperdyadics}

If we use a faster increasing sequence of branching factors than that satisfying \eqref{definingsequencegrowthrate}, we must exploit some extra property of our construction, which is not always present in general sets. Here, we rely on a \emph{uniform mass distribution} between scales. Given the uniformity assumption, the lengths can decrease as fast as desired. We utilize the multi-scale set of dyadic cubes $\DQ^d$ and $\DR^d$ introduced in Section \ref{sec:Dyadics}.

\begin{lemma} \label{uniformMassFrostman}
	Let $\mu$ be a measure supported on a set $E$. Suppose that
    \begin{enumerate}
    	\item \label{discreteBound} For any $Q \in \DQ_k^d$, $\mu(Q) \lesssim l_k^s$.
    	\item \label{controlledScale} For each $R \in \DR_{k+1}^d$, $\#[\DQ_{k+1}^d(E(l_k) \cap R)] \lesssim 1$.
    	\item \label{uniformDist} For any $R \in \DR_{k+1}^d$ with parent cube $Q \in \DQ_k^d$, $\mu(R) \lesssim (1/M_{k+1})^d \cdot \mu(Q)$.
    \end{enumerate}
	Then $\mu$ is a Frostman measure of dimension $s$.
\end{lemma}
\begin{proof}
	We establish the general bound $\mu(Q) \lesssim l(Q)^s$ for all cubes $Q$ by separating our analysis into two different cases:
	\begin{itemize}
		\item Suppose there is $k$ with $r_{k+1} \leq l(Q) \leq l_k$. Then $\#(\DR_{k+1}^d(Q(r_{k+1}))) \lesssim (l/r_{k+1})^d$. Properties \ref{discreteBound} and \ref{uniformDist} imply each of these cubes has measure at most $O( (r_{k+1}/l_k)^d l_k^s)$, so we obtain that
    	\[ \mu(Q) \lesssim (l/r_{k+1})^d (r_{k+1}/l_k)^d l_k^s = l^d / l_k^{d-s} \lesssim l^s. \]

    \item Suppose there exists $k$ with $l_k \leq l \leq r_k$. Then $\#[\DR_k^d(Q(r_k))] \lesssim 1$. Combining this with Property \ref{uniformDist} gives $\#[\DQ_k^d(Q(r_k) \cap E(l_k))] \lesssim 1$. This and Property \ref{discreteBound} then shows
    \[ \mu(Q) \lesssim l_k^s \leq l^s. \]
	\end{itemize}
	We have addressed all cases, so $\mu$ is a Frostman measure of dimension $s$.
\end{proof}

This theorem is commonly used here when dealing with multi-scale dyadic fractal constructions, whose character is summarized in the following Theorem.

\begin{theorem} \label{TheConstructionTheorem}
	let $E = \bigcap E_k$, where $\{ E_k \}$ is a nested family of subsets of $\RR^d$, such that $E_k$ is $\DQ_k$ discretized for each $k$. Suppose that 
	\begin{enumerate}
		\item \label{SingleSelection} For each $Q \in \DQ_k(E_k)$, there exists a set $\mathcal{R}_Q \subset \DR_{k+1}(Q)$ such that $\#(\mathcal{R}_Q) \geq (1/2) \cdot \#(\DR_{k+1}(Q))$, and
		\[ \# \left( \DQ_{k+1}(R \cap E_{k+1}) \right) = \begin{cases} 1 & : R \in \mathcal{R}_Q, \\ 0 & : R \not \in \mathcal{R}_Q. \end{cases} \]

		\item \label{RapidDecrease} For any $\varepsilon > 0$, $N_1 \dots N_k \lesssim_\varepsilon N_{k+1}^\varepsilon$.

		\item \label{ChangeofScales} There is $0 < s \leq 1$ such that for any $\varepsilon > 0$, $N_{k+1} \lesssim_\varepsilon M_{k+1}^{(1 + \varepsilon)/s}$.
	\end{enumerate}
	Then $E$ has Hausdorff dimension $sd$.
\end{theorem}

\begin{remark}
	Note that Property \ref{RapidDecrease} is essentially the opposite of \eqref{definingsequencegrowthrate}.
\end{remark}

\begin{proof}
	Consider the function $w: \DQ^d \to [0,\infty)$ defined with respect to the sequence $\{ E_k \}$, which generates the canonical measure $\mu$ supported on $E$ using the mass distribution principle. For each $R \in \DR_{k+1}(E_k)$ with parent cube $Q \in \DQ_k(E_k)$, Property \ref{SingleSelection} shows
	\begin{equation} \label{equation6250234923} w(R) \leq (2/M_{k+1}^d) f(Q). \end{equation}
	Properties \ref{RapidDecrease} and \ref{ChangeofScales} together imply that for each $\varepsilon > 0$,
	\begin{equation} \label{equation363491043214} 1/M_{k+1} \lesssim_\varepsilon r_{k+1}^{1 - \varepsilon} \lesssim_\varepsilon l_{k+1}^{s(1  - 2 \varepsilon)}. \end{equation}
	If $Q \in \DQ_{k+1}(E_{k+1})$ has parent cubes $R \in \DR_{k+1}(E_k)$ and $Q^* \in \DQ_k(E_k)$, then by \eqref{equation6250234923}, \eqref{equation363491043214}, and the fact that $w(Q^*) \leq 1$,
	\begin{equation} \label{equation120492309562} w(Q) = w(R) \leq (2/M_{k+1}^d) w(Q^*) \leq (2/M_{k+1}^d) \lesssim_{\varepsilon} l_{k+1}^{ds(1 - 2\varepsilon)}. \end{equation}
	If $Q$ is a cube with length $l_k$, then $\#(\DQ_k^d(Q(l_k))) = O_d(1)$, which combined with \eqref{massdisslowerbound}, shows that
	\begin{equation} \label{equation69009230919} \mu(Q) \lesssim_\varepsilon l_{k+1}^{ds(1-2\varepsilon)}. \end{equation}
	Property \ref{controlledScale} of Lemma \ref{uniformMassFrostman} is implied by Property \ref{SingleSelection}.  Equation \eqref{equation69009230919} is a form of Property \ref{discreteBound} in Lemma \ref{uniformMassFrostman}. Together with \eqref{massdisslowerbound}, \eqref{equation6250234923} implies that the canonical measure $\mu$ satisfies Property \ref{uniformDist} of Lemma \ref{uniformMassFrostman}. Thus all assumptions of Lemma \ref{uniformMassFrostman} are satisfied, and so we conclude $\mu$ is a Frostman measure of dimension $ds(1 - 2\varepsilon)$ for each $\varepsilon > 0$. Taking $\varepsilon \to 0$, and applying Frostman's lemma, we conclude that the Hausdorff dimension of the support of $\mu$, which is a subset of $E$, is greater than or equal to $ds$. For each $k$, and $\varepsilon > 0$,
	\[ \#(\DQ_{k+1}(E_{k+1})) \leq \#(\DR_{k+1}(E_k)) \leq (1/M_{k+1})^d \lesssim_\varepsilon l_{k+1}^{ds - \varepsilon}. \]
	Taking $k \to \infty$, and then $\varepsilon \to 0$, we conclude that
	\[ \hausdim(E) \leq \lowminkdim(E) \leq ds. \]
	Since we already know $\hausdim(E) \geq ds$, this completes the proof.
\end{proof}

\chapter{Related Work}
\label{ch:RelatedWork}

Here, we discuss the main papers which influenced our results. In particular, the work of Keleti on translate avoiding sets, Fraser and Pramanik's work on sets avoiding smooth configurations, and Math\'{e}'s result on sets avoiding algebraic varieties. 

\begin{comment}
and Schmerkin's result on sets with large Fourier dimension avoiding smooth configurations.
\end{comment}

\section{Keleti: A Translate Avoiding Set}

In \cite{KeletiDimOneSet}, Keleti constructs a set $X \subset [0,1]$ with Hausdorff dimension such that for each $t \neq 0$, $X$ intersects $t + X$ in at most one place. The set $X$ is then said to \emph{avoid translates}. This paper contains the core idea behind the \emph{discretization} method adapted in Fraser and Pramanik's paper. We also adapt this technique in our paper, which makes the result of interest.

\begin{lemma}
    Let $X$ be a set. Then $X$ avoids translates if and only if there do not exists values $x_1 < x_2 \leq x_3 < x_4$ in $X$ with $x_2 - x_1 = x_4 - x_3$. In particular, a set $X$ avoids translates if and only if it avoids the four point configuration
    \[ \C = \{ (x_1, x_2, x_3, x_4): x_1 < x_2 \leq x_3 < x_4,\ x_2 - x_1 = x_4 - x_3 \}. \]
\end{lemma}
\begin{proof}

    Suppose $(t + X) \cap X$ contains two points $a < b$. Without loss of generality, we may assume that $t > 0$. If $a \leq b - t$, then the equation
    \[ a - (a - t) = t = b - (b - t) \]
    shows that the tuple $(a-t,a,b-t,b)$ lies in $\C$. We also have
    \[ (b - t) - (a - t) = b - a, \]
    so if $a - t < b - t \leq a < b$, then $(a-t,b-t,a,b) \in \C$. This covers all possible cases. Conversely, if there are $x_1 < x_2 \leq x_3 < x_4$ in $X$ with
    \[ x_2 - x_1 = t = x_4 - x_3, \]
    then $X + t$ contains $x_2 = x_1 + (x_2 - x_1)$ and $x_4 = x_3 + (x_4 - x_3)$.
\end{proof}

%\footnote{We always assume $L_n/L_{n+1}$ is an integer so that intervals in $\mathcal{B}(L_n)$ are either almost disjoint from intervals in $\mathcal{B}(L_{n+1})$ or contained completely within such an interval}

The basic, but fundamental idea of the interval dissection technique is to introduce memory into Cantor set constructions. Keleti constructs a nested family of discrete sets $\{ X_k \}$, with $X_k$ a $\DQ_k$ discretized set, and with $X = \bigcap X_k$. The sequence $\{ N_k \}$ will be specified later, but each $N_k$ will be a multiple of 10. We initialize $X_0 = [0,1]$. The novel feature of the argument is to incorporate a queuing procedure into the construction, i.e. the algorithm incorporates a list of intervals (known as a queue) that changes over the course of the algorithm, as we remove intervals from the front of the queue, and add intervals to the back of the queue. The queue initially just contains the interval $[0,1]$, and we let $X_0 = [0,1]$. To construct the sequence $\{ X_k \}$, Keleti iteratively performs the following procedure:
\begin{algorithm}[H]
    \begin{algorithmic}%[1]
        \caption{Construction of the Sets $\{ X_k \}$:}
        \State{Set $k = 0$.}
        \MRepeat
            \State{Take off an interval $I$ from the front of the queue.}

            \MForAll{\ $J \in \DQ_k(X_k)$:}
                \State{Order the intervals in $\DQ_{k+1}(J)$ as $J_1, \dots, J_N$.}

                \State{{\bf If} $J \subset I$, add all intervals $J_i$ to $X_{k+1}$ with $i \equiv 0$ modulo 10.}
                \State{{\bf Else} add all $J_i$ with $i \equiv 5$ modulo 10.}
            \EndForAll
            \State{Add all intervals in $\DQ_{k+1}^d$ to the end of the queue.}
            \State{Increase $k$ by 1.}
        \EndRepeat   
    \end{algorithmic}
\end{algorithm}

Each iteration of the algorithm produces a new set $X_k$, and so leaving the algorithm to repeat infinitely produces a sequence $\{ X_k \}$ whose intersection is $X$.

\begin{lemma}
    The set $X$ is translate avoiding.
\end{lemma}
\begin{proof}
    If $X$ is not translate avoiding, there is $x_1 < x_2 \leq x_3 < x_4$ with $x_2 - x_1 = x_4 - x_3$. Since $l_k \to 0$, there is a suitably large integer $N$ such that $x_1$ is contained in an interval $I \in \DQ_N$ not containing $x_2,x_3$, or $x_4$. At stage $N$ of the algorithm, the interval $I$ is added to the end of the queue, and at a much later stage $M$, the interval $I$ is retrieved. Find the start points $x_1^\circ, x_2^\circ$, $x_3^\circ, x_4^\circ \in l_M \mathbf{Z}$ to the intervals in $\DQ_M$ containing $x_1$, $x_2$, $x_3$, and $x_4$. Then we can find $n$ and $m$ such that $x_4^\circ - x_3^\circ = (10n)l_M$, and $x_2^\circ - x_1^\circ = (10m + 5)l_M$. In particular, this means that $|(x_4^\circ - x_3^\circ) - (x_2^\circ - x_1^\circ)| \geq 5L_M$. But
    \begin{align*}
        |(x_4^\circ - x_3^\circ) - (x_2^\circ - x_1^\circ)| &= |[(x_4^\circ - x_3^\circ) - (x_2^\circ - x_1^\circ)] - [(x_4 - x_3) - (x_2 - x_1)]|\\
        &\leq |x_1^\circ - x_1| + \dots + |x_4^\circ - x_4| \leq 4 L_M
    \end{align*}
    which gives a contradiction.
\end{proof}

It is easy to see from the algorithm that
\[ \# (\DQ_k(X_k)) = (N_k/10) \cdot \#(\DQ_{k-1}(X_{k-1})). \]
Closing the recursive definition shows
\[ \#(\DQ_k(X_k)) = \frac{1}{10^k l_k}. \]
In particular, this means $|X_k| = 1/10^k$, so $|X|$ has Lebesgue measure zero regardless of how we choose the parameters $\{ N_k \}$. 

\begin{theorem}
    For some sequence $\{ N_k \}$, the set $X$ has full Hausdorff dimension.
\end{theorem}
\begin{proof}
    Set $N_k = 10 M_k$ for each $k \geq 1$, then one sees that for each $R \in \DR_k(X_k)$, $\#(\DQ_{k+1}(R \cap X_{k+1})) = 1$. Thus Property \ref{SingleSelection} of Theorem \ref{TheConstructionTheorem} is satisfied. Furthermore, Property \ref{ChangeofScales} of Theorem \ref{TheConstructionTheorem} is satisfied with $s = 1$. We conclude that if $N_1 \dots N_k \lesssim_\varepsilon N_{k+1}^\varepsilon$, then all assumptions of Theorem \ref{TheConstructionTheorem} is satisfied, and we conclude $X$ is a set with full Hausdorff dimension. This is true, for instance, if we set $N_k = 2^{2^{\lfloor k \log k \rfloor}}$.
\end{proof}

The most important feature of Keleti's argument is his reduction of a non-discrete configuration avoidance problem to a sequence of discrete avoidance problems on cubes. Let us summarize the result of Keleti's discrete argument in a lemma.

\begin{lemma} \label{KeletiDiscreteLemma}
    Let $T_1, T_2 \subset \RR$ be disjoint, $\DQ_k$ discretized sets. If $M_{k+1} = N_{k+1}/10$, then we can find $S_1 \subset T_1$ and $S_2 \subset T_2$ such that
    \begin{enumerate}
        \item[(i)] For each $k$, $S_k$ is a $\DQ_{k+1}$ discretized subset of $T_k$.
        \item[(ii)] If $x_1 \in S_1$ and $x_2,x_3,x_4 \in S_2$, then $x_2 - x_1 \neq x_4 - x_3$.
        \item[(iii)] For each $i$, and each cube $R \in \DR_{k+1}(S_i)$, $\#(\DQ_{k+1}(R \cap S_i)) = 1$.
    \end{enumerate}
\end{lemma}
Property (i) of Lemma \ref{KeletiDiscreteLemma} allows Keleti to apply his argument iteratively at each dyadic scale. Property (ii) implies Keleti obtains a configuration avoiding set by iterative the argument infinitely many times. And the reason why Keleti obtains a set with full Hausdorff dimension, relating back to Property \ref{ChangeofScales} in Theorem \ref{TheConstructionTheorem}, is because of Property (iii), which allows us to select $N_{k+1} \lesssim M_{k+1}$.

Of course, it is not possible to extend the discrete solution of the configuration argument to general configurations; this part of Keleti's method strongly depends on the arithmetic structure of the configuration. The iterative application of a discrete solution, however, can be applied in generality. Combined with Theorem \ref{TheConstructionTheorem}, this technique gives a powerful method to reduce configuration avoidance problems about Hausdorff dimension to discrete avoidance problems on cubes.

\section{Fraser/Pramanik: Smooth Configurations}

Inspired by Keleti's result, in \cite{MalabikaRob}, Pramanik and Fraser obtained a generalization of the queue method which allows one to find sets avoiding $n+1$ point configurations given by the zero sets of smooth functions, i.e.
\[ \C = \{ (x_1, \dots, x_n) \in \C^n[0,1]^d : f(x_0, \dots, x_n) = 0 \}, \]
under mild regularity conditions on the function $f: \C^n[0,1]^d \to [0,1]^m$.

\begin{theorem}[Pramanik and Fraser] \label{pramanikandfrasertheorem}
    Fix $m \leq d(n-1)$. Consider a countable family of $C^2$ functions $\{ f_k : [0,1]^{dn} \to [0,1]^m \}$ such that for each $k$, $Df_k$ has full rank at any $(x_1, \dots, x_n) \in \C^n[0,1]^d$ where $f_k(x_1, \dots, x_n) = 0$. Then there exists a set $X \subset \RR^d$ with Hausdorff dimension $m/(n-1)$ such that $X$ avoids the configuration
    \[ \C = \bigcup_k \{ (x_1, \dots, x_n) \in \C^n(\RR^d): f_k(x_1, \dots, x_n) = 0 \}. \]
\end{theorem}

\begin{remark}
    For simplicity, we only prove the result for a single function, rather than a countable family of functions. The only major difference between the two approaches is the choice of scales we must choose later on in the argument, and a slight modification of the queuing argument.
\end{remark}

Just like Keleti, Pramanik and Fraser begin by solving a discrete variant of the configuration problem in Theorem \ref{pramanikandfrasertheorem}, which they can then iteratively apply at each scale. In the discrete setting, rather than making a linear shift in one of the variables, as in Keleti's approach, Pramanik and Fraser must utilize the smoothness properties of the function which defines the configuration to find large sets. Corollary \ref{PramanikFraserBuildingBlockLemma} gives the discrete result that Pramanik and Fraser utilize in a queuing construction, analogous to the queueing method of Keleti, to find a set $X$ satisfying the conclusions of Theorem \ref{pramanikandfrasertheorem}. Here, a multi-scale approach proves useful, so we utilize the family of cubes $\DQ$ and $\DR$, assuming the existence of two sequences of branching factors $\{ N_k \}$ and $\{ M_k \}$ which we will specify later on in the argument.

\begin{lemma} \label{Lemma315091513}
    Fix $n > 1$. Let $T \subset [0,1]^d$ and $T' \subset [0,1]^{(n-1)d}$ be $\DQ_k$ discretized sets. Let $B \subset T \times T'$ be $\DQ_{k+1}$ discretized. Then there exists a $\DQ_{k+1}$ discretized set $S \subset T$, and a $\DQ_{k+1}$ discretized set $B' \subset T'$, such that
    \begin{enumerate}
        \item \label{dimensionReductionProperty} $(S \times T') \cap B \subset S \times B'$.

        \item \label{bigProperty} For every $Q \in \DQ_k^d$, there exists $\mathcal{R}(Q) \subset \DR_{k+1}^d(Q)$, such that
        \[ \#(\mathcal{R}(Q)) \geq (1/2) \cdot \#(\DR_{k+1}^d(Q)), \]
        and for each $R \in \DR_{k+1}^d(Q)$,
        \[ \#(\DQ_{k+1}(R)) = \begin{cases} 1 &: R \in \mathcal{R}(Q), \\ 0 &: R \not \in \mathcal{R}(Q). \end{cases} \]

        \item \label{BBoundProperty} $\#(\DQ_{k+1}(B')) \leq 2 (N_1 \dots N_k)^d \left( M_{k+1}/N_{k+1} \right)^d \cdot \#(\DQ_{k+1}(B))$.
    \end{enumerate}
\end{lemma}
\begin{proof}
    Fix $Q_0 \in \DQ_k(T)$. For each $R \in \DR_{k+1}(Q_0)$, define a \emph{slab} $S[R] = R \times T'$, and for each $Q \in \DQ_{k+1}(Q_0)$, define a \emph{wafer} $W[Q] = Q \times T'$. We say a wafer $W[Q]$ is \emph{good} if
    \begin{equation} \label{equation10291095429062}
        \#(\DQ_{k+1}(W[Q] \cap B)) \leq (2/N_{k+1}^d) \cdot \#(\DQ_{k+1}(B)).
    \end{equation}
    Then at most $N_{k+1}^d/2$ wafers are bad. We call a slab \emph{good} if it contains a wafer which is good. Since a slab is the union of $(N_{k+1}/M_{k+1})^d$ wafers, at most $M_{k+1}^d /2 = (1/2) \cdot \#(\DR_{k+1}(Q_0))$ slabs are bad. Thus if we set
    \[ \mathcal{R}(Q_0) = \{ R \in \DR_{k+1}(Q_0) : S[R]\ \text{is good} \}, \]
    then
    \begin{equation} \label{equation24016590369046}
        \#(\mathcal{R}(Q_0)) \geq (1/2) \cdot \#(\DR_{k+1}(Q_0)).
    \end{equation}
    For each $R \in \mathcal{R}(Q_0)$, we pick $Q_R \in \DQ_{k+1}(R)$ such that $W[Q_R]$ is good, and define
    \[ S = \bigcup \{ Q_R : R \in \mathcal{R}(Q_0) \}. \]
    Equation \eqref{equation24016590369046} implies $S$ satisfies Property \ref{bigProperty}.

    Let $B'$ be the union of all cubes $Q' \in \DQ_{k+1}(T')$ such that there is $Q \in \DQ_{k+1}(S)$ with $Q \times Q' \in \DQ_{k+1}(B)$. By definition, Property \ref{dimensionReductionProperty} is then satisfied. For each $Q \in \DQ_{k+1}(S)$, $W[Q]$ is good, so \eqref{equation10291095429062} implies
    \[ \# \{ Q' : Q \times Q' \in \DQ_{k+1}(B) \} \leq (2/N_{k+1}^d) \cdot \#(\DQ_{k+1}(B)). \]
    But $\#(\DQ_{k+1}(S)) \leq \#(\DR_{k+1}(T)) \leq (1/r_{k+1})^d = (N_1 \dots N_k)^d M_{k+1}^d$, so
    \begin{align*}
        \#(\DQ_{k+1}(B')) &\leq \#(\DQ_{k+1}(S))[(2/N_{k+1}^d) \cdot \#(\DQ_{k+1}(B))]\\
        &\leq 2(N_1 \dots N_k)^d (M_{k+1}/N_{k+1})^d \#(\DQ_{k+1}(B)),
    \end{align*}
    which establishes Property \ref{BBoundProperty}.
\end{proof}

We apply the lemma recursively $n-1$ times to continually reduce the dimensionality of the avoidance problem we are considering. Eventually, we obtain the case where $n = 0$, and then avoiding the configuration is easy.

\begin{lemma} \label{Lemma1209410535}
    Fix $n > 1$. Let $T \subset [0,1]^d$ be $\DQ_k$ discretized, and let $B \subset T$ be $\DQ_{k+1}$ discretized. Suppose
    \[ \# \DQ_{k+1}(B) \leq \left[ C \cdot 2^{n-1} (N_1 \dots N_k)^{d(n-1)} (M_{k+1}/N_{k+1})^{d(n-1)} \right] (1/l_{k+1})^{dn - m}. \]
    and
    \begin{equation} \label{equation903103513095}
        N_{k+1} \geq \left[ C \cdot 2^n (N_1 \dots N_k)^{2dn} \right]^{1/m} M_{k+1}^{d(n-1)/m}.
    \end{equation}
    Then there exists a $\DQ_{k+1}$ discretized set $S \subset T$ such that
    \begin{enumerate}
        \item $S \cap B = \emptyset$.
        \item \label{badsetproperty5} For each $Q_0 \in \DQ_k(T)$, there is $\mathcal{R}(Q_0) \subset \DR_{k+1}(Q_0)$ with
        \[ \#(\mathcal{R}(Q_0)) \geq (1/2) \#(\DR_{k+1}(Q_0)), \]
        such that for each $R \in \DR_{k+1}(Q_0)$,
        \[ \#(\mathcal{Q}_{k+1}(R \cap S)) = \begin{cases} 1 &: R \in \mathcal{R}(Q_0), \\ 0 &: R \not \in \mathcal{R}(Q_0). \end{cases} \]
    \end{enumerate}
\end{lemma}
\begin{proof}
    For each $Q_0 \in \DQ_k(T)$, we set
    \[ \mathcal{R}(Q_0) = \{ R \in \DR_{k+1}(Q_0) : \#(\DQ_{k+1}(R \cap B)) \leq (2/M_{k+1}^d) \cdot \#(\DQ_{k+1}(B)) \}. \]
    Since $\DR_{k+1}(Q_0) = M_{k+1}^d$,
    \[ \#(\mathcal{R}(Q_0)) \geq \#(\DR_{k+1}(Q_0)) - (M_{k+1}^d/2) \geq (1/2) \cdot \#(\DR_{k+1}(T)). \]
    Now \eqref{equation903103513095} implies that for each $R \in \mathcal{R}(Q_0)$,
    \begin{align*}
        \#(\DQ_{k+1}(R \cap B)) &\leq (2/M_{k+1}^d) \cdot \#(\DQ_{k+1}(B))\\
        &\leq (2/M_{k+1}^d) \left(C \cdot 2^{n-1} (N_1 \dots N_k)^{2dn} (M_{k+1}/N_{k+1})^{d(n-1)} \right).\\
        &= \left[ 2^n C (N_1 \dots N_k)^{2dn} \right] \left( M_{k+1}^{d(n-2)} / N_{k+1}^{m-d} \right)\\
        &< (N_{k+1}/M_{k+1})^d\\
        &= \DQ_{k+1}(R).
    \end{align*}
    Thus for each $R \in \mathcal{R}(Q_0)$, we can find $Q_R \in \DQ_{k+1}(R)$ such that $Q_R \cap B = \emptyset$. And so if we set
    \[ S = \bigcup \{ Q_R : R \in \mathcal{R}(Q_0), Q_0 \in \DQ_k(T) \}, \]
    then (A) and (B) are satisfied.
\end{proof}

\begin{corollary} \label{PramanikFraserBuildingBlockLemma}
    Let $f: [0,1]^{dn} \to [0,1]^m$ be $C^2$, and have full rank at every point $(x_1, \dots, x_n) \in \C^n(\RR^d)$ such that $f(x_1, \dots, x_n) = 0$. Then there exists a universal constant $C$ depending only on $f$ such that, if \eqref{equation903103513095} is satisfied, then for any disjoint, $\DQ_k$ discretized sets $T_1, \dots, T_n \subset [0,1]^d$, we can find $\DQ_{k+1}$ discretized sets $S_1 \subset T_1, \dots, S_n \subset T_n$ such that
    \begin{enumerate}
        \item If $x_1 \in S_1, \dots, x_n \in S_n$, then $f(x_1, \dots, x_n) \neq 0$.
        \item For each $k$, and for each $Q_0 \in \DQ_k(T_k)$, there is $\mathcal{R}(Q_0) \subset \mathcal{R}_{k+1}(Q_0)$ with
        \[ \#(\mathcal{R}(Q_0)) \geq (1/2) \cdot \#(\mathcal{R}_{k+1}(Q_0)), \]
        and for each $R \in \DR_{k+1}(Q_0)$,
        \[ \#(\mathcal{Q}_{k+1}(R \cap S)) = \begin{cases} 1 &: R \in \mathcal{R}(Q_0), \\ 0 &: R \not \in \mathcal{R}(Q_0). \end{cases} \]
    \end{enumerate}
\end{corollary}
\begin{proof}
    Since $f$ is $C^2$ and has full rank on the set
    \[ V(f) = \{ (x_1, \dots, x_n) \in \C^n(\RR^d) : f(x_1, \dots, x_n) = 0 \}, \]
    the implicit function theorem implies $V(f)$ is a smooth manifold of dimension $nd - m$ in $\RR^{dn}$, and so the co-area formula implies the existence of a constant $C$ such that for each $k$,
    \[ \# \{ Q \in \DQ_k^{dn} : Q \cap V(f) \neq \emptyset \} \leq C/l_k^{dn-m}. \]
    To apply Lemma \ref{Lemma315091513} and \ref{Lemma1209410535}, we set
    \[ B = \# \{ Q \in \DQ_k^{dn} : Q \cap V(f) \neq \emptyset \}. \]
    Applying Lemma \ref{Lemma315091513} iteratively $n-1$ times, then finishing with an application of Lemma $\ref{Lemma1209410535}$ constructs the required sets $S_1, \dots, S_n$.
\end{proof}

Just like in Keleti's proof, Pramanik and Fraser's technique applies a discrete result, Corollary \ref{PramanikFraserBuildingBlockLemma}, iteratively at many scales, with the help of a queuing process, to obtain a high dimensional set avoiding the zeros of a function. We construct a nested family $\{ X_k : k \geq 0 \}$ of $\DQ_k$ discretized sets, converging to a set $X$, which we will show is translate avoiding. We initialize $X_0 = [0,1]$. Our queue shall consist of $n$ tuples of disjoint intervals $(T_1, \dots, T_n)$, all of the same length, which initially consists of all possible tuples of intervals in $\DQ_1^d([0,1]^d)$. To construct the sequence $\{ X_k \}$, we perform the following iterative procedure:
\begin{algorithm}[H]
    \begin{algorithmic}
        \caption{Construction of the Sets $\{ X_k \}$}
        \State{Set $k = 0$}
        \MRepeat
            \State{Take off an $n$ tuple $(T_1', \dots, T_n')$ from the front of the queue}
            \State{Set $T_i = T_i' \cap X_k$ for each $i$}
            \State{Apply Corollary \ref{PramanikFraserBuildingBlockLemma} to the sets $T_1, \dots, T_d$, obtaining $\DQ_{k+1}$ discretized sets $S_1, \dots, S_n$ satisfying Properties (A), (B), and (C) of that Lemma.}
            \State{Set $X_{k+1} = X_k - \bigcup_{i = 1}^n T_i - S_i$.}
            \State{Add all $n$ tuples of disjoint cubes $(T_1', \dots, T_n')$ in $\DQ_{k+1}^d(X_{k+1})$ to the back of the queue.}
            \State{Increase $k$ by 1.}
        \EndRepeat   
    \end{algorithmic}
\end{algorithm}

\begin{lemma}
    The set $X$ constructed by the procedure avoids the configuration
    \[ \C = \{ (x_1, \dots, x_n) \in \C^n(\RR^d) : f(x_1, \dots, x_n) = 0 \}. \]
\end{lemma}
\begin{proof}
    Suppose $x_1, \dots, x_n \in X$ are distinct. Then at some stage $k$, $x_1, \dots, x_n$ lie in disjoint cubes $T_1', \dots, T_n' \in \DQ_k^d(X_k)$, for some large $k$. At this stage, $(T_1', \dots, T_n')$ is added to the back of the queue, and therefore, at some much later stage $N$, the tuple $(T_1', \dots, T_n')$ is taken off the front. Sets $S_1 \subset T_1', \dots, S_n \subset T_n'$ are constructed satisfying Property (A) of Corollary \ref{PramanikFraserBuildingBlockLemma}. Since $x_1, \dots, x_n \in X$, we must have $x_i \in S_i$ for each $i$, so $f(x_1, \dots, x_n) \neq 0$.
\end{proof}

What remains is to show that for some sequence of parameters $\{ N_k \}$ and $\{ M_k \}$ satisfying \eqref{equation903103513095}, we can apply Theorem \ref{TheConstructionTheorem}. The conclusion of Corollary \ref{PramanikFraserBuildingBlockLemma} shows Property \ref{SingleSelection} of Theorem \ref{TheConstructionTheorem} is always satisfied. If we set
\[ N_{k+1} = \left\lceil \left[ C \cdot 2^n (N_1 \dots N_k)^{2dn} \right]^{1/m} M_{k+1}^{d(n-1)/m} \right\rceil, \]
then \eqref{equation903103513095} holds, so we can apply Lemma \ref{pramanikandfrasertheorem} at each scale. Properties \ref{RapidDecrease} and \ref{ChangeofScales} of Theorem \ref{TheConstructionTheorem} are satisfied with $s = m/d(n-1)$ if $N_1 \dots N_k \lesssim_\varepsilon M_{k+1}^\varepsilon$ for each $\varepsilon > 0$. This is true, for instance, if we set $M_k = 2^{\lfloor 2^{k \log k} \rfloor}$, and Theorem \ref{TheConstructionTheorem} then shows the resultant set $X$ has Hausdorff dimension $m/(n-1)$.

\section{Math\'{e}: Polynomial Configurations}

Math\'{e}'s result \cite{Mathe} constructs sets avoiding low degree algebraic hypersurfaces.

\begin{theorem}[Math\'{e}] \label{mathemainresult}
    For each $k$, let $f_k: \RR^{n_k d} \to \RR$ be a rational coefficient polynomial with degree at most $m$. Then there exists a set $X \subset [0,1]^d$ with Hausdorff dimension $d/m$ which avoids the configuration
    \[ \C = \bigcup_k \{ (x_1, \dots, x_n) \in \C^{n_k}(\RR^d) : f_k(x_1, \dots, x_n) = 0 \}. \]
\end{theorem}

Originally, Math\'{e}'s result does not explicitly use a discretization method analogous to Keleti and Pramanik and Fraser, but his proof strategy can be reconfigured to work in this setting. For the purpose of brevity, we do not carry out the complete argument, merely giving the discretization method below. By first trying to avoid the zero sets of the partial derivatives of the function $f$, one can reduce to the case where a partial derivative of $f$ is non-vanishing on the $\DQ_k$ discretized sets we start with. The key technique is that $f$ maps discrete lattices of points to a discrete, `one dimensional lattice' in $\RR^d$, and the degree of the polynomial gives us the difference in lengths between the two lattices. This idea was present in Keleti's work, and one can view Math\'{e}'s result as a generalization along these lines. We revisit this idea in Chapter \ref{ch:Conclusions}. As in Pramanik and Fraser's result, we utilize the multi-scale dyadic notations $\DQ$ and $\DR$, with an implicitly chosen sequence $\{ N_k \}$ and $\{ M_k \}$.

\begin{theorem}
    Let $f: [0,1]^{dn} \to \RR$ be a rational coefficient polynomial of degree $m$, and disjoint, $\DQ_k$ discretized sets $T_1, \dots, T_n \subset [0,1]^d$, such that $\inf |\partial_1 f| \neq 0$ on $T_1 \times \dots \times T_n$. Then there exists a constant $C$, depending only on $f$, such that if
    \begin{equation} \label{equation124426034990370935} N_{k+1} \geq C \cdot (N_1 \dots N_k)^{m-1} \cdot M_{k+1}^m, \end{equation}
    then there exists $\DQ_{k+1}$ discretized sets $S_1 \subset T_1, \dots, S_n \subset T_n$ such that
    \begin{enumerate}
        \item $f(x) \neq 0$ for $x \in S_1 \times \dots \times S_n$.
        \item For each $i$, and for each $R \in \DR_{k+1}^d(T_i)$, $\#(\DQ_{k+1}^d(R \cap S_i)) = 1$.
    \end{enumerate}
\end{theorem}
\begin{proof}
    Without loss of generality, by considering an appropriate integer multiple of $f$, we may assume $f$ has integer coefficients. Let $\mathbf{A} \subset (r_{k+1} \cdot \mathbf{Z})^d$. Since $f$ has degree $m$, $f(\mathbf{A}) \subset r_{k+1}^m \cdot \mathbf{Z}$. Suppose that $c_0 \leq |\partial_1 f| \leq C_0$ on $T_1 \times \dots \times T_n$. Then the mean value theorem guarantees that if $\delta \leq r_{k+1}$, then
    \[ c_0 \delta \leq |f(a + \delta e_1) - f(a)| \leq C_0 \delta. \]
    Pick $\varepsilon \in (0,1)$ small enough that $\varepsilon/c_0 < (1 - \varepsilon)/C_0$. If
    \begin{equation} \label{equation920340923}
        l_{k+1} \leq [(1-\varepsilon)/C_0 - \varepsilon/c_0] r_{k+1}^m
    \end{equation}
    we can find $\delta \in l_{k+1} \cdot \mathbf{Z}$ is such that
    \[ [\varepsilon/c_0] r_{k+1}^m \leq \delta \leq [(1 - \varepsilon)/C_0] r_{k+1}^m. \]
    Equation \eqref{equation920340923} is guaranteed by \eqref{equation124426034990370935} for a sufficiently large constant $C$. We then find $d(f(\mathbf{A} + \delta e_1), r_{k+1}^m \cdot \ZZ) \geq \varepsilon r_{k+1}^m$. For $i > 1$, define
    \[ S_i = \bigcup [a_1, a_1 + l_{k+1}] \times \dots \times [a_d, a_d + l_{k+1}], \]
    where $a = (a_1, \dots, a_d)$ ranges over all start points to intervals
    \[ [a_1,a_1 + r_{k+1}] \times \dots \times [a_d,a_d + r_{k+1}] \in \DR_{k+1}^d(T_i). \]
    Similarly, define
    \[ S_1 = \bigcup [a_1 + \delta e_1, a_1 + \delta e_1 + s] \times [a_2, a_2 + l_{k+1}] \times \dots \times [a_d, a_d + l_{k+1}], \]
    where $a$ ranges over all start points to intervals
    \[ [a_1,a_1 + r_{k+1}] \times \dots \times [a_d,a_d + r_{k+1}] \in \DR_{k+1}^d(T_1). \]
    %
    %     % [(1 - \varepsilon)/C_1 - \varepsilon/C_0] r_{k+1}^m > l_{k+1}
    %
    Because $f$ is $C^1$, we can find $C_2$ such that if $x, y \in \RR^d$ satisfy $|x_i - y_i| \leq t$ for all $i$, then $|f(x) - f(y)| \leq C_2 t$. But then for a sufficiently large constant $C$, \eqref{equation124426034990370935} implies that
    \[ d(f(S_1 \times \dots \times S_n), r_{k+1}^m \cdot \ZZ) \geq \varepsilon r_{k+1}^m - C_2 l_{k+1} \geq (\varepsilon/2) r_{k+1}^m. \]
    In particular, this implies Property (A).
\end{proof}

\chapter{Avoiding Rough Sets}
\label{ch:RoughSets}

In the previous chapter, we saw that many authors have considered the pattern avoidance problem for configurations $\C$ which take the form of many general classes of smooth shapes; in Math\'{e}'s work, $\C$ can take the form of an algebraic variety of low degree, and in Pramanik and Fraser's work, $\C$ can take the form of a smooth manifold. In this chapter, we consider the pattern avoidance problem for an even more general class of `rough' patterns, that are the countable union of sets with controlled lower Minkowski dimension. The results of this chapter, and the applications obtained from these results detailed in the following chapter, have been accepted into the Springer series \emph{Harmonic Analysis and Applications}, in a paper entitled \emph{Large Sets Avoiding Rough Patterns}. A preprint \cite{roughSetsAvoidingPatterns} is also available on the ArXiv.
\begin{theorem}\label{mainTheorem}
	Let $s \geq d$, and suppose $\C \subset \C^n(\RR^d)$ is the countable union of precompact sets, each with lower Minkowski dimension at most $s$. Then there exists a set $X \subset [0,1]^d$ with Hausdorff dimension at least $(nd - s)/(n-1)$ avoiding $\C$.
\end{theorem}

\begin{remarks}
	\
	\begin{enumerate}
		\item[1.] When $s < d$, avoiding the configuration $\C$ is trivial. If we define $\pi: \C^n(\RR^d) \to \RR^d$ by $\pi(x_1, \dots, x_n) = x_1$, then the set $X = [0,1]^d - \pi(\C)$ is full dimension and avoids $\C$. Note that obtaining a full dimensional set in the case $s = d$, however, is still interesting.

		\item[2.] Theorem \ref{mainTheorem} is trivial when $s = dn$, since we can set $X = \emptyset$. We will therefore assume that $s < dn$ in our proof of the theorem.

		\item[3.] Let $f: \RR^{dn} \to \RR^m$ be a $C^1$ map such that $f$ has full rank at any point $(x_1, \dots, x_n) \in \C^n(\RR^d)$ with $f(x_1, \dots, x_n) = 0$. If we set
		\[ \C = \{ x \in \C^n(\RR^d) : f(x) = 0 \}, \]
		Then $\C$ is a $C^1$ submanifold of $\C^n(\RR^d)$ of dimension $nd - m$. A submanifold of Euclidean space is $\sigma$ compact, so we can write $\C = \bigcup K_i$, where each $K_i$ is a compact set. Applying Theorem \ref{ManifoldDimensionThm} shows $\lowminkdim(K_i) \leq nd - m$ for each $i$, so we can apply Theorem \ref{mainTheorem} with $s = nd - m$ to yield a set in $\RR^d$ with Hausdorff dimension at least
		\[ \frac{nd - s}{n-1} = \frac{m}{n-1}. \]
		This recovers Theorem \ref{pramanikandfrasertheorem}, making Theorem \ref{mainTheorem} a generalization of Pramanik and Fraser's result.

		\item[4.] Since Theorem \ref{mainTheorem} does not require any regularity assumptions on the set $\C$, it can be applied in contexts that cannot be addressed using previous methods in the literature. Two such applications, new to the best of our knowledge, have been recorded in Chapter \ref{ch:Applications}; see Theorems \ref{sumset-application} and \ref{C1IsoscelesThm} there.
	\end{enumerate}
\end{remarks}

Like with the results considered in the last chapter, we construct the set $X$ in Theorem \ref{mainTheorem} by repeatedly applying a discrete avoidance result at the scales corresponding to cubes $\DQ^d$, constructing a Frostman measure with equal mass at intermediary scales, and applying Lemma \ref{uniformMassFrostman}. However, our method has several innovations that simplify the analysis of the resulting set $X = \bigcap X_k$ than from previous results. In particular, through a probabilistic selection process we are able to use a simplified queuing technique then that used in \cite{MalabikaRob} and \cite{KeletiDimOneSet}, that required storage of data from each step of the iterated construction to be retrieved at a much later stage of the construction process.

%The details of a single step of this construction are described in Section \ref{discretesection}. In Section \ref{discretizationsection}, we explain how the branching factors $\{ N_k \}$ must be chosen, complete the construction of $X$, and prove $X$ avoids the configuration $\C$. In Section \ref{dimensionsection} we analyze the size of $X$ and show it has the Hausdorff dimension guaranteed by the conclusion of Theorem \ref{mainTheorem}.

\section{Avoidance at Discrete Scales}\label{discretesection}

In this section we describe a method for avoiding a discretized version of $\C$ at a single scale. We apply this technique in Section \ref{discretizationsection} at many scales to construct a set $X$ avoiding $\C$ at all scales.
%This single scale avoidance technique is the core building block of our construction, and the efficiency with which we can avoid $\C$ at a single scale has direct consequences on the Hausdorff dimension of the set $X$ obtained in Theorem \ref{mainTheorem}.
In the discrete setting, $\C$ is replaced by a union of cubes in $\DQ^{dn}_{k+1}$ denoted by $B$. We say a cube $Q = Q_1 \times \dots \times Q_n \in \DQ^{dn}_{k+1}$ is \emph{strongly non-diagonal} if the $n$ cubes $Q_1, \dots, Q_n$ are distinct. Given a $\DQ_k$ discretized set $T \subset \RR^d$, our goal is to construct a $\DQ_{k+1}$ discretized set $S \subset T$, such that $\DQ_{k+1}^{dn}(F^n)$ does not contain any strongly non-diagonal cubes of $\DQ_{k+1}^{dn}(B)$.

%To ensure the set $X$ obtained in Theorem \ref{mainTheorem} has large Hausdorff dimension regardless of the rapid decay of scales used in the construction of $X$, it is crucial that $F$ is uniformly distributed among $\DR_{k+1}^d(T)$ so that we can apply Lemma \ref{uniformMassFrostman}. This is what Lemma \ref{discretelemma} achieves, given that $N_{k+1}$ is suitably large in comparsion to $M_{k+1}$, i.e. so that \eqref{rBound} holds.

\begin{lemma} \label{discretelemma}
	Fix $k$, $s \in [1,dn)$, and $\varepsilon \in [0,(dn-s)/2)$. Let $T \subset \RR^d$ be a nonempty, $\DQ_k$ discretized set, and let $B \subset \RR^{dn}$ be a nonempty $\DQ_{k+1}$ discretized set such that
	\[ \#(\DQ_{k+1}(B)) \leq N_{k+1}^{s + \varepsilon}. \]
	Then there exists a constant $C(s,d,n) > 0$%\footnote{The proof will show that we can choose $C(s,d,n) \sim_d 4^{\frac{1}{dn - s}}$, which explodes as $s \to dn$.}
	, depending only on $s$, $d$, and $n$, such that, provided
	% C(s,d,n) \geq 4d, 2^{\frac{n+1}{dn - t}}
	\begin{equation} \label{rBound}
		N_{k+1} \geq C(s,d,n) \cdot M_{k+1}^{\frac{d(n-1)}{dn - s - \varepsilon}},
	\end{equation}
	then there is a $\DQ_{k+1}$ discretized set $S \subset T$ satisfying the following three properties:
	\begin{enumerate}
		\item\label{avoidanceItem} For any collection of $n$ distinct cubes $Q_1, \dots, Q_n \in \DQ_{k+1}(S)$,
		\[ Q_1 \times \dots \times Q_n \not \in \DQ_{k+1}(B). \]

		\item\label{nonConcentrationItem} For each $Q \in \DQ_k(T)$, there exists $\mathcal{R}_Q \subset \DR_{k+1}(Q)$ such that
		\[ \#(\mathcal{R}_Q) \geq \frac{\#(\DR_{k+1}(Q))}{2}, \]
		and if $R \in \DR_{k+1}(Q)$,
		\[ \#(\DQ_{k+1}(R \cap S)) = \begin{cases} 1 & : R \in \mathcal{R}_Q \\ 0 & : R \not \in \mathcal{R}_Q. \end{cases} \]
	\end{enumerate}
\end{lemma}

\begin{proof}
	For each $R \in \DR_{k+1}(T)$, pick $Q_R$ uniformly at random from $\DQ_{k+1}(R)$; these choices are independent as $R$ ranges over $\DR_{k+1}(T)$. Define
	\[ A = \bigcup \left\{ Q_R \setcolon R \in \DR_{k+1}(T) \right\}, \]
	and
	\[ \mathcal{K}(A) = \{ K \in \DQ_{k+1}(B) \cap \DQ_{k+1}(A^n) \setcolon \text{$K$ strongly non-diagonal} \}. \]
	The sets $A$ and $\mathcal{K}(A)$ are random, in the sense that they depend on the random variables $\{ Q_R \}$. Define
	\begin{equation} \label{defnOfF}
		S(A) = \bigcup \Big[ \DQ_{k+1}(A) - \{ \pi(K) \setcolon K \in \mathcal{K}(A) \} \Big],
	\end{equation}
	where $\pi \colon \RR^{dn} \to \RR^d$ is the projection map $(x_1, \dots, x_n) \mapsto x_1$, for $x_i \in \RR^d$.

	Given any strongly non-diagonal cube $K = K_1 \times \cdots \times K_n \in \DQ_{k+1}(B)$, either $K \not \in \DQ_{k+1}(A^n)$, or $K \in \DQ_{k+1}(A^n)$. If the former occurs then $K \not \in \DQ_{k+1}(S(A))$ since $S(A) \subset A$, so $\DQ_{k+1}(S(A)^n) \subset \DQ_{k+1}(A^n)$. If the latter occurs then $K \in \mathcal{K}(A)$, and since $\pi(K) = K_1$, $K_1 \not \in \DQ_{k+1}(S(A))$. In either case, $K \not \in \DQ_{k+1}(S(A)^n)$, so $S(A)$ satisfies Property \ref{avoidanceItem}. By construction, we know that for each $R \in \DR_{k+1}(T)$,
	\begin{equation} \label{equation8423246093490} \#(\DQ_{k+1}(S(A) \cap R)) \leq \#(\DQ_{k+1}(A \cap R)) = 1. \end{equation}
	To conclude that $S(A)$ satisfies Property \ref{nonConcentrationItem}, it therefore suffices to show that
	\[ \# (\mathcal{K}(A)) \leq (1/2) \cdot M_{k+1}^d. \]
	We now show this is true with non-zero probability.

	For each cube $Q \in \DQ_{k+1}(T)$, there is a unique `parent' cube $R \in \DR_{k+1}(T)$ such that $Q \subset R$. Note that \eqref{rBound} implies, if $C(s,d,n) \geq 4d$, that $N_{k+1} \geq 4d \cdot M_{k+1}$. Since $Q_R$ is chosen uniformly from $\DQ_{k+1}(R)$, for any $Q \in \DQ_{k+1}(R)$,
	\[ \prob(Q \subset A) = \prob(Q_R = Q) = \left( \DQ_{k+1}(R) \right)^{-1} = (M_{k+1}/N_{k+1})^d. \]
	Suppose $K_1, \dots, K_n \in \DQ_{k+1}(T)$ are distinct cubes. If the cubes have distinct parents in $\DR_{k+1}^d$, we can apply the independence of the random cubes $\{ Q_R \}$ to conclude that
	\[ \prob(K_1, \dots, K_n \subset A) = (M_{k+1}/N_{k+1})^{dn}. \]
	If the cubes $K_1, \dots, K_n$ do not have distinct parents, \eqref{equation8423246093490} shows
	\[ \prob(K_1, \dots, K_n \subset A) = 0. \]
	In either case, we conclude that
	\begin{equation}\label{jointprob}
	\prob(K_1, \dots, K_n \subset A) \leq (M_{k+1}/N_{k+1})^{dn}.
	\end{equation}
	Let $K = K_1 \times \dots \times K_n$ be a strongly non-diagonal cube in $\DQ_{k+1}(B)$. We deduce from \eqref{jointprob} that
	\begin{equation}\label{probaKSubsetUn}
		\prob(K \subset A^n) = \prob(K_1, \dots, K_n \subset A) \leq (M_{k+1}/N_{k+1})^{dn}.
	\end{equation}
	If
	\[ C(s,d,n) \geq 4^{\frac{1}{dn - s}} \geq 2^{\frac{1}{dn - s - \varepsilon}}, \]
	by \eqref{probaKSubsetUn}, linearity of expectation, and \eqref{rBound}, we conclude
	\begin{align*}
		\expect(\#(\mathcal{K}(A))) &= \sum_{K \in \DQ_{k+1}(B)} \prob(K \subset A^n)\\
		&\leq \#(\DQ_{k+1}(B)) \cdot (M_{k+1}/N_{k+1})^{dn}\\
		&\leq M_{k+1}^{dn} / N_{k+1}^{dn - s - \varepsilon}\\
		&\leq (1/2) \cdot M_{k+1}^d.
	\end{align*}
	In particular, there exists at least one (non-random) set $A_0$ such that
	\begin{equation}\label{KU0Small}
		\#(\mathcal{K}(A_0)) \leq \expect(\# (\mathcal{K}(A))) \leq (1/2) \cdot M_{k+1}^d.
	\end{equation}
	In other words, $S(A_0) \subset A_0$ is obtained by removing at most $(1/2) \cdot M_{k+1}^d$ cubes in $\B^d_s$ from $A_0$. For each $Q \in \B_l^d(T)$, we know that $\# \DQ_{k+1}(Q \cap A_0) = M_{k+1}^d$. Combining this with \eqref{KU0Small}, we arrive at the estimate 
	\begin{align*}
		\# \DQ_{k+1}(Q \cap S(A_0)) &= \DQ_{k+1}(Q \cap A_0) - \# \{ \pi(K) \setcolon K \in \mathcal{K}(A_0), \pi(K) \in A_0 \}\\
		&\geq \DQ_{k+1}(Q \cap A_0) - \#(\mathcal{K}(A_0))\\
		&\geq M_{k+1}^d - (1/2) \cdot M_{k+1}^d \geq (1/2) \cdot M_{k+1}^d
	\end{align*}  
	Thus, $S(A_0)$ satisfies Property \ref{nonConcentrationItem}. Setting $S = S(A_0)$ completes the proof.
\end{proof}

\begin{remarks}
	\
	\begin{enumerate}
		\item[1.] While Lemma \ref{discretelemma} uses probabilistic arguments, the proof of the lemma is still constructive. In particular, one can find a suitable $S$ constructively by checking every possible choice of $A$ (there are finitely many) to find one particular choice $A_0$ which satisfies \eqref{KU0Small}, and then defining $S$ by \eqref{defnOfF}. Thus the set we obtain in Theorem \ref{mainTheorem} exists by purely constructive means.
		
		\item[2.] As with the proofs in Chapter \ref{ch:RelatedWork}, the fact that we can choose
		\[ N_{k+1} \sim M_{k+1}^{(1 - \varepsilon)/t} \]
		where $t = (dn - s)/d(n-1)$, implies we should be able to iteratively apply Lemma \ref{discretelemma} to obtain a set with Hausdorff dimension $(dn - s)/(n-1)$.
	\end{enumerate} 
\end{remarks}

\section{Fractal Discretization}\label{discretizationsection}

In this section we construct the set $X$ from Theorem \ref{mainTheorem} by applying Lemma \ref{discretelemma} at many scales. Let us start by fixing a strong cover of $\C$, as well as the branching factors $\{ N_k \}$, that we will work with in the sequel.

\begin{lemma}\label{coveringLemma}
	Let $\C \subset \C^n(\RR^d)$ be a countable union of bounded sets with Minkowski dimension at most $s$, and let $\epsilon_k \searrow 0$ with $\epsilon_k < (dn - s)/2$ for all $k$. Then there exists a choice of branching factors $\{ N_k : k \geq 1 \}$, and a sequence of sets $\{ B_k : k \geq 1 \}$, with $B_k \subset \RR^{dn}$ for all $k$, such that
	\begin{enumerate}
		\item\label{StrongCoverProperty} \emph{Strong Cover}: The interiors $\{ B_k^\circ \}$ of the sets $\{ B_k \}$ form a strong cover of $\C$.

		\item\label{DiscretenessProperty} \emph{Discreteness}: For all $k \geq 1$, $B_k$ is a $\DQ_k$ discretized subset of $\RR^{dn}$.

		\item\label{SparsityProperty} \emph{Sparsity}: For all $k \geq 1$, $\DQ_k(B_k) \leq N_k^{s+\epsilon_k}$.

		\item \label{RapidDecayProperty} \emph{Rapid Decay}: For all $\varepsilon > 0$ and $k \geq 1$, $N_k$ is a power of two such that
		\[ N_1 \dots N_{k-1} \lesssim_\varepsilon N_k^\varepsilon, \]
		and for all $k \geq 1$,
		\[ N_k \geq C(s,d,n). \]
	\end{enumerate}
\end{lemma}
\begin{proof}
	We can write $\C = \bigcup_{i = 1}^\infty Y_i$, with $\lowminkdim(Y_i) \leq s$ for each $i$. Let $\{ i_k \}$ be a sequence of integers that repeats each integer infinitely often. The branching factors $\{ N_k \}$ and sets $\{ B_k \}$ are defined inductively. Suppose that the lengths $N_1, \ldots, N_{k-1}$ have been chosen. Since $\lowminkdim(Y_{i_k}) < s + \varepsilon_k$, there are arbitrarily small lengths $l \leq l_{k-1}$ such that if $N_k = l_{k-1}/l$,
	% N_k \geq (N_1 \dots N_{k-1})^{2s/\varepsilon_k}
	\begin{equation} \label{coveringOfBdnlZk}
		\# (\DQ_k(Y_{i_k}(l))) \leq (1/l)^{s + (\varepsilon_k/2)}.
	\end{equation}
	% l^{\varepsilon_k/2} \leq l_{k-1}^{s + \varepsilon_k}
	In particular, we can choose $l$ small enough that $l \leq l_{k-1}^{2s/\varepsilon + 2}$, which together with \eqref{coveringOfBdnlZk} implies
	\begin{equation} \label{equation789891491}
		\# (\DQ_k(Y_{i_k})(l)) \leq N_k^{s + \varepsilon_k}.
	\end{equation}
	We can certainly also choose $l$ small enough that
	\[ N_k \geq \max(C(s,d,n),(N_1 \dots N_{k-1})^{1/\varepsilon_k}). \]
	We know that $l_{k-1}$ is a power of two, so we can ensure $N_k$ is a power of two. We then set $l_k = l$. With this choice, Property \ref{RapidDecayProperty} is satisfied. And if $B_k = \bigcup \DQ_k(Y_{i_k}(l_k))$, then this choice of $B_k$ clearly satisfies Property \ref{DiscretenessProperty}. And Property \ref{SparsityProperty} is precisely Equation \eqref{equation789891491}.

	It remains to verify that the sets $\{ B_k^\circ \}$ strongly cover $\C$. Fix a point $z \in \C$. Then there exists an index $i$ such that $z \in Y_i$, and there is a subsequence $k_1, k_2, \dots$ such that $i_{k_j} = i$ for each $j$. But then $z \in Y_i \subset B_{i_{k_j}}^\circ$, so $z$ is contained in each of the sets $B_{i_{k_j}}^\circ$, and thus $z \in \limsup B_i^\circ$. Thus Property \ref{StrongCoverProperty} is proved.
\end{proof}

To construct $X$, we consider a nested, decreasing family of sets $\{ X_k \}$, where each $X_k$ is an $l_k$ discretized subset of $\RR^d$. We then set $X = \bigcap X_k$. The goal is to choose $X_k$ such that $X_k^n$ does not contain any {\it strongly non diagonal} cubes in $B_k$.

\begin{lemma} \label{stronglydiagonal}
	For each $k$, let $B_k \subset \RR^{dn}$ be a $\DQ_k$ discretized set, such that the interiors $\{ B_k^\circ \}$ strongly cover $\C$. For each index $k$, let $X_k$ be a $\DQ_k$ discretized set such that $\DQ_k(X_k^n) \cap \DQ_k(B_k)$ contains no strongly non diagonal cubes. If $X = \bigcap X_k$, then $X$ avoids $B$.
\end{lemma}
\begin{proof}
	Let $x = (x_1, \dots, x_n) \in \C$ be a point such that $x_1, \dots, x_n$ are distinct. Define
	\[ \Delta = \{ (y_1, \dots, y_n) \in \RR^{dn} \setcolon \text{there exists $i \neq j$ such that $y_i = y_j$} \}. \]
	Then $d(\Delta,x) > 0$, where $d$ is the Hausdorff distance between $\Delta$ and $x$. Since $\{ B_k \}$ strongly covers $\C$, there is a subsequence $\{ k_m \}$ such that $x \in B_{k_m}^\circ$ for every index $m$. Since $l_k$ converges to 0 and thus $l_{k_m}$ converges to $0$, if $m$ is sufficiently large then $\sqrt{dn} \cdot l_{k_m} < d(\Delta,x)$. Note that $\sqrt{dn} \cdot l_{k_m}$ is the diameter of a cube in $\DQ_{k_m}$. For such a choice of $m$, any cube $Q \in \DQ_{k_m}^d$ which contains $x$ is strongly non-diagonal. Thus $x \in B_{k_m}^\circ$. Since $X_{k_m}$ and $B_{k_m}$ share no cube which contains $x$, this implies $x \not \in X_{k_m}$. In particular, this means $x \not \in X^n$.
\end{proof}

All that remains is to apply the discrete lemma to choose the sequence $\{ X_k \}$. Given $\C$, apply Lemma \ref{coveringLemma} to choose a sequence $\{ N_k \}$ and a sequence $\{ B_k \}$. We recursively define the sequence $\{ X_k \}$. Set $X_0 = [0,1]^d$. Then, Property \ref{RapidDecayProperty} of Lemma \ref{coveringLemma} implies
\[ \#(\DQ_{k+1}(B_{k+1})) \leq N_{k+1}^{s + \varepsilon_k} \]
Let $M_{k+1}$ be the largest power of two smaller than
\begin{equation} \label{equation9249815935} \left( \frac{N_{k+1}}{C(s,d,n)} \right)^{\frac{dn - s -\varepsilon_k}{d(n-1)}} \end{equation}
Then $1 \leq M_{k+1} \leq N_{k+1}$, and \eqref{rBound} is satisfied. Set $B = B_{k+1}$, and $T = X_k$. Then we can apply Lemma \ref{discretelemma} to find a set $S$ satisfying all Properties of that Lemma, and set $X_{k+1} = S$.

Property \ref{avoidanceItem} of Lemma \ref{discretelemma} together with Lemma \ref{stronglydiagonal} implies that the set $X = \bigcap X_k$ avoids $\C$. Property \ref{nonConcentrationItem} of Lemma \ref{discretelemma} shows that the sequence $\{ X_k \}$ satisfies Property \ref{SingleSelection} of Theorem \ref{TheConstructionTheorem}. Property \ref{RapidDecayProperty} of Lemma \ref{coveringLemma} implies Property \ref{RapidDecrease} of Theorem \ref{TheConstructionTheorem} is satisfied. And by definition, i.e. the choice of $\{ M_k \}$ as given by \eqref{equation9249815935}, we know
\begin{equation} \label{equation1403463468957983} N_k \leq M_k^{\frac{d(n-1)}{dn - s - \varepsilon_k}} \lesssim_\varepsilon M_k^{\frac{1 + \varepsilon}{t}} \end{equation}
where $t = (nd - s)/d(n-1)$. \eqref{equation1403463468957983} is a form of Property \ref{ChangeofScales} of Theorem \ref{TheConstructionTheorem}. Thus all assumptions of Theorem \ref{TheConstructionTheorem} are satisfied, and so we find $X$ has Hausdorff dimension $(nd - s)/(n-1)$, completing the proof of Theorem \ref{mainTheorem}.

\chapter{Applications} \label{ch:Applications}

As discussed in the introduction, Theorem \ref{mainTheorem} generalizes Theorems 1.1 and 1.2 from \cite{MalabikaRob}. In this chapter, we present two applications of Theorem \ref{mainTheorem} in settings where previous methods do not yield any results.

\section{Sum Sets Avoiding Specified Sets}

\begin{theorem} \label{sumset-application} 
	Let $Y \subset \RR^d$ be a countable union of sets with lower Minkowski dimension at most $t$. Then there exists a set $X \subset \RR^d$ with Hausdorff dimension at least $d - t$ such that $X + X$ is disjoint from $Y$.
\end{theorem}
\begin{proof}
	Define $\C = \C_1 \cup \C_2$, where
	\[ \C_1 = \{ (x,y) \setcolon x + y \in Y \} \quad \text{and} \quad \C_2 = \{ (x,y) \setcolon y \in Y/2 \}. \]
	Since $Y$ is a countable union of sets with lower Minkowski dimension at most $t$, $\C$ is a countable union of sets with lower Minkowski dimension at most $d + t$. Applying Theorem \ref{mainTheorem} with $n = 2$ and $s = d + t$ produces a set $X \subset \RR^d$ with Hausdorff dimension $2d  - (d + t) = d - t$ such that $(x,y) \not \in \C$ for all $x,y \in X$ with $x \neq y$. We claim that $X+ X$ is disjoint from $Y$. To see this, first suppose $x, y \in X$, $x \neq y$. Since $X$ avoids $Z_1$, we conclude that $x + y \not \in Y$. Suppose now that $x = y \in X$. Since $X$ avoids $Z_2$, we deduce that $X \cap (Y/2) = \emptyset$, and thus for any $x \in X$, $x + x = 2x \not \in Y$. This completes the proof.
\end{proof}

\section{Subsets of Lipschitz curves avoiding isosceles triangles}

In \cite{MalabikaRob}, Fraser and the second author prove that there exists a set $S \subset [0,1]$ with dimension $\log_3 2$ such that for any simple $C^2$ curve $\gamma \colon [0,1] \to \RR^n$ with bounded non-vanishing curvature, $\gamma(S)$ does not contain the vertices of an isosceles triangle. Our method enables us to obtain a result that works for Lipschitz curves with small Lipschitz constants. The dimensional bound that we provide is slightly worse than \cite{MalabikaRob} ($1/2$ instead of $\log_3 2$), and the set we obtain only works for a single Lipschitz curve, not for many curves simultaneously.

\begin{theorem} \label{C1IsoscelesThm}
	Let $f \colon [0,1] \to \RR^{n-1}$ be Lipschitz with \[ \| f \|_{\text{Lip}}  := \sup \bigl\{|f(x) - f(y)|/|x-y| : x, y \in [0,1], x \ne y   \bigr\} < 1. \]  Then there is a set $X \subset [0,1]$ of Hausdorff dimension $1/2$ so that the set
	\[ \{(t,f(t)) \setcolon t\in X\} \]
	does not contain the vertices of an isosceles triangle.
\end{theorem}

\begin{corollary} \label{C1IsoscelesCor}
	Let $f\colon [0,1] \to \RR^{n-1}$ be $C^1$.  Then there is a set $X \subset [0,1]$ of Hausdorff dimension $1/2$ so that the set
	\[ \{(t,f(t)) \setcolon t\in X\} \]
	does not contain the vertices of an isosceles triangle.
\end{corollary} 
\begin{proof}[Proof of Corollary \ref{C1IsoscelesCor}]
	The graph of any $C^1$ function can be locally expressed, after possibly a translation and rotation, as the graph of a Lipschitz function with small Lipschitz constant. In particular, there exists an interval $I\subset[0,1]$ of positive length so that the graph of $f$ restricted to $I$, after being suitably translated and rotated, is the graph of a Lipschitz function $g\colon [0,1] \to \RR^{n-1}$ with Lipschitz constant at most $1/2$. Since isosceles triangles remain invariant under these transformations, the corollary is a consequence of Theorem \ref{C1IsoscelesThm}.  
\end{proof} 

\begin{proof}[Proof of Theorem \ref{C1IsoscelesThm}]
	Set
	\begin{equation} \label{def-Z}
		\C = \left\{ (x_1,x_2,x_3) \in \C^3[0,1] \setcolon \; \begin{array}{c}
			\text{The points $p_j = (x_j,f(x_j))$ form}\\
			\text{the vertices of an isosceles triangle}
		\end{array} \right\}.
	\end{equation} 
	In the next lemma, we show $\C$ has lower Minkowski dimension at most two. By Theorem \ref{mainTheorem}, there is a set $X \subset [0,1]$ of Hausdorff dimension $1/2$ so that $X$ avoids $\C$. This is precisely the statement that for each $x_1,x_2,x_3\in X$, the points $(x_1,f(x_1)),\ (x_2,f(x_2))$, and $(x_3,f(x_3))$ do not form the vertices of an isosceles triangle. 
\end{proof}

\begin{lemma}
	Let $f\colon [0,1] \to \RR^{n-1}$ be Lipschitz with $\| f \|_{\text{Lip}} < 1$. Then the set $\C$ given by \eqref{def-Z} satisfies $\upminkdim(\C) \leq 2$.
\end{lemma}
\begin{proof}
	First, notice that three points $p_1,p_2,p_3 \in \RR^n$ form an isosceles triangle, with $p_3$ as the apex, if and only if $p_3 \in H_{p_1,p_2}$, where
	\begin{equation} \label{def-H}  H_{p_1,p_2} = \left\{ x \in \RR^n \setcolon \left( x - \frac{p_1 + p_2}{2} \right) \cdot (p_2 - p_1) = 0 \right\}. \end{equation} 
	To prove $\C$ has Minkowski has dimension at most two, it suffices to show the set
	\[ W = \left\{ x \in [0,1]^3 \setcolon p_3 = (x_3,f(x_3)) \in H_{p_1, p_2} \right\} \]
	has upper Minkowski dimension at most 2. This is because $\C$ is covered by three copies of $W$, obtained by permuting coordinates. We work with the family of dyadic cubes $\DD^n$. To bound the upper Minkowski dimension of $W$, we prove the estimate
	\begin{equation}\label{boundOnWCoveringNumber}
		\# \bigl(\mathcal \DD_k(W(1/2^k)) \bigr) \leq C k 4^k \quad \text{ for all } k \geq 1,  
	\end{equation}  
	where $C$ is a constant independent of $k$. Then for any $\varepsilon > 0$, \eqref{boundOnWCoveringNumber} implies that for suitably large $k$,
	\[ \# \bigl(\mathcal \DD_k(W(1/2^k)) \bigr) \leq 2^{(2 + \varepsilon)k}. \]
	Since $\varepsilon$ was arbitrary, this shows $\upminkdim(W) \leq 2$.

	To establish \eqref{boundOnWCoveringNumber}, we write
	\begin{equation}\label{deltaCoveringWSum}
		\# \bigl(\mathcal \DD_k(W(1/2^k)) \bigr) = \sum_{m = 0}^{2^k}\ \sum_{\substack{I_1, I_2 \in \DD_k[0,1]\\d(I_1,I_2) = m/2^k}} \# \left( \DD_k(W(1/2^k) \cap (I_1 \times I_2 \times [0,1]) \right).
	\end{equation}
	Our next task is to bound each of the summands in \eqref{deltaCoveringWSum}.
	 Let $I_1, I_2 \in \DD_k[0,1]$, and let $m = 2^k \cdot d(I_1,I_2)$. Let $x_1$ be the midpoint of $I_1$, and $x_2$ the midpoint of $I_2$. Let $(y_1,y_2,y_3) \in W \cap (I_1 \times I_2 \times [0,1])$. Then it follows from \eqref{def-H} that 
	\[ \left( y_3 - \frac{y_1 + y_2}{2} \right) \cdot (y_2 - y_1) + \left( f(y_3) - \frac{f(y_2) + f(y_1)}{2} \right) \cdot (f(y_2) - f(y_1)) = 0. \]
	We know $|x_1 - y_1|, |x_2 - y_2| \leq 1/2^{k+1}$, so
	\begin{align} \label{xyDiff}
		&\left| \left( y_3 - \frac{y_1 + y_2}{2} \right) (y_2 - y_1) - \left( y_3 - \frac{x_1 + x_2}{2} \right) (x_2 - x_1) \right| \nonumber\\
		&\ \ \ \ \ \leq \frac{|y_1 - x_1| + |y_2 - x_2|}{2} |y_2 - y_1| + \Big( |y_1 - x_1| + |y_2 - x_2| \Big) \left| y_3 - \frac{x_1 + x_2}{2} \right|\\
		&\ \ \ \ \ \leq (1/2^{k+1}) \cdot 1 + (1/2^k) \cdot 1 \leq 3/2^{k+1}. \nonumber
	\end{align}
	Conversely, $|f(x_1) - f(y_1)|, |f(x_2) - f(y_2)| \leq 1/2^{k+1}$ because $\| f \|_{\text{Lip}} \leq 1$, and a similar calculation yields
	\begin{align} \label{fnDiff}
	\begin{split}
		&\Big| \left( f(y_3) - \frac{f(y_1) + f(y_2)}{2} \right) \cdot (f(y_2) - f(y_1))\\
		&\ \ \ \ \ - \left( f(y_3) - \frac{f(x_1) + f(x_2)}{2} \right) \cdot (f(x_2) - f(x_1)) \Big|\leq 3/2^{k+1}.
	\end{split}
	\end{align}
	Putting \eqref{xyDiff} and \eqref{fnDiff} together, we conclude that
	\begin{align} \label{hyperplanethick}
	\begin{split}
		&\Big| \left( y_3 - \frac{x_1 + x_2}{2} \right) (x_2 - x_1)\\
		&\ \ \ \ \ + \left( f(y_3) - \frac{f(x_2) + f(x_1)}{2} \right) \cdot (f(x_2) - f(x_1)) \Big| \leq 3/2^k.
	\end{split}
	\end{align}
	Since $|(x_2-x_1,f(x_2)-f(x_1))| \geq |x_2-x_1| \geq m/2^k$, we can interpret \eqref{hyperplanethick} as saying the point $(y_3, f(y_3))$ is contained in a $3/k$ thickening of the hyperplane $H_{(x_1,f(x_1)), (x_2,f(x_2))}$. Given another value $y' \in W \cap (I_1 \cap I_2 \cap [0,1])$, it satisfies a variant of the inequality \eqref{hyperplanethick}, and we can subtract the difference between the two inequalities to conclude
	\begin{equation} \label{diffinequality}
		\left| \left( y_3 - y_3' \right) (x_2 - x_1) + (f(y_3) - f(y_3')) \cdot (f(x_2) - f(x_1)) \right| \leq 6/2^k.
	\end{equation}
	The triangle difference inequality applied with \eqref{diffinequality} implies
	\begin{align} \label{yylowbound}
	\begin{split}
		(f(y_3) - f(y_3')) \cdot (f(x_2) - f(x_1)) &\geq |y_3 - y_3'||x_2-x_1| - 6/2^k\\ &= \frac{(m+1) \cdot |y_3 - y_3'| - 6}{2^k}.
	\end{split}
	\end{align}
	Conversely,
	\begin{align} \label{yyupbound}
	\begin{split}
		(f(y_3) - f(y_3')) \cdot (f(x_2) - f(x_1)) &\leq \| f \|_{\text{Lip}}^2 \cdot |y_3 - y_3'| |x_2 - x_1| \\ &=  \| f \|_{\text{Lip}}^2 \cdot (m+1)/2^k \cdot |y_3 - y_3'|.
	\end{split}
	\end{align}
	Combining \eqref{yylowbound} and \eqref{yyupbound} and rearranging, we see that
	\begin{equation}\label{y3minusY3Prime}
		|y_3 - y_3'| \leq \frac{6}{(m+1)(1 - \| f \|_{\text{Lip}}^2)}  \lesssim\frac{1}{m+1},
	\end{equation} 
	where the implicit constant depends only on $\| f \|_{\text{Lip}}$.  We conclude that
	\begin{equation}\label{coveringNumberBoundLargeK}
		\# \DQ_k(W(1/2^k) \cap (I_1 \times I_2 \times [0,1])) \lesssim \frac{2^k}{m+1},
	\end{equation} 
	which holds uniformly over any value of $m$.

	We are now ready to bound the sum from \eqref{deltaCoveringWSum}. Note that for each value of $m$, there are at most $2^{k+1}$ pairs $(I_1,I_2)$ with $d(I_1,I_2) = m/2^k$. Indeed, there are $2^k$ choices for $I_1$ and then at most two choices for $I_2$. Equation  \eqref{coveringNumberBoundLargeK} shows 
	\begin{align*}
		\# \DQ_k(W(1/2^k)) &= \sum_{m = 0}^{2^k} \sum_{\substack{I_1, I_2 \in \DQ_k[0,1]\\d(I_1,I_2) = m/2^k}} \# \DQ_k(W(1/2^k) \cap (I_1 \times I_2 \times [0,1]))\\
		&\lesssim 4^k \sum_{m = 0}^{2^k} \frac{1}{m+1} \lesssim k/4^k.
	\end{align*}
	In the above inequalities, the implicit constants depend on $\| f \|_{\text{Lip}},$ but they are independent of $k$. This establishes \eqref{boundOnWCoveringNumber} and completes the proof.
\end{proof}

%\include{lowrank}
%\include{fourierDimension}
%% The following is a directive for TeXShop to indicate the main file
%%!TEX root = diss.tex

\chapter{Future Work}
\label{ch:Conclusions}

To conclude this thesis, we sketch some ideas developing the theory of `rough sets avoiding patterns', which we introduced in Chapter \ref{ch:RoughSets}. Section 6.1 attempts to exploit additional geometric information about certain rough configurations to find sets with large Hausdorff dimension avoiding patterns, and Section 6.2 finds configuration avoiding sets supported a measure with large Fourier decay.

\section{Low Rank Avoidance}

One way we can extend the results of Chapter \ref{ch:RoughSets} is to utilize additional geometric structure of particular rough configurations $\C$ to obtain larger avoiding sets. Recall that in Chapter \ref{ch:RoughSets}, we studied the avoidance problem for configurations with low Minkowski dimension. This condition means precisely that these configurations are efficiently covered by cubes at all scales. The idea of this section is to study configurations which are efficiently covered by other families of geometric objects at all scales. Here, we study the simple setting where our set is efficiently covered by families of thickened hyperplanes or thickened lines. We note that a set $E$ is efficiently covered by a family of thickened parallel hyperplanes, at each scale of thickening, if and only, for a linear transformation $M$ with that hyperplane as a kernel, $M(E)$ has low Minkowski dimension.

\begin{theorem} \label{theorem9063909014901}
    Let $\C \subset \C(\RR)$ be the countable union of sets $\{ \C_i \}$ such that
    \begin{itemize}
        \item For each $i$, there exists $n_i$ such that $\C_i \subset \C^{n_i}(\RR)$.

        \item There exists an integer $m_i > 0$ and $s_i \in [0,m_i)$, together with a full-rank rational-coefficient linear transformation $M_i: \RR^{n_i} \to \RR^{m_i}$ such that $M_i(\C_i)$ is a bounded subset of $\RR^{m_i}$ with lower Minkowski dimension at most $s_i$.
    \end{itemize}
    Then there exists a set $X \subset [0,1]$ avoiding $\C$ with Hausdorff dimension at least
    \[ \inf_i \left( \frac{m_i - s_i}{m_i} \right). \]
\end{theorem}

\begin{remarks}
    \
    \begin{enumerate}
        \item[1.] A useful feature of this method is that the resulting set does not depend on the number of points in a configuration, i.e. we do not need to restrict ourselves to an $n$ point configuration for some fixed integer $n$. This is a feature only shared by Math\'{e}'s result, Theorem \ref{mathemainresult} in Section 3.3. We exploit this feature later on in this section to find large subsets avoiding a countable family of equations with arbitrarily many variables.

        \item[2.] It might be expected, based on the result of Theorem \ref{mainTheorem}, that one should be able to obtain a set $X \subset [0,1]$ avoiding $\C$ with Hausdorff dimension
        \[ \inf_i \left( \frac{m_i - s_i}{m_i - 1} \right), \]
        whenever $s_i \geq 1$ for all $i$. We plan to study whether Theorem \ref{theorem9063909014901} can be improved to give this bound in the near future.

        \item[3.] Compared to Theorem \ref{mainTheorem}, this result only applies in the one-dimensional configuration avoidance setting. We also plan to study higher dimensional analogues to this theorem, when $d > 1$.
    \end{enumerate}
\end{remarks}

For the purpose of brevity, here we only describe a solution to the discretized version of the problem. This can be fleshed out into a full proof of Theorem \ref{theorem9063909014901} by techniques analogous to those given in Chapters \ref{ch:RelatedWork} and \ref{ch:RoughSets}. Thus we discuss a single linear transformation $M: \RR^{dn} \to \RR^m$, and try to avoid a discretized version of a low dimensional set.

Before we describe the discretized result, let us simplify the problem slightly. Since our transformation $M$ has full rank, we may find indices
\[ i_1, \dots, i_m \in \{ 1, \dots, n \} \]
such that the transformation $M$ is invertible when restricted to the span of $\{ e_{i_1}, \dots, e_{i_m} \}$. By an affine change of coordinates in the range of $M$, which preserves the Minkowski dimension of any set, we may assume without loss of generality that $M(e_{i_j}) = e_j$ for each $1 \leq j \leq m$.

\begin{theorem} \label{theorem059891891829}
    Fix $s \in [0,m)$ and $\varepsilon \in [0, (m-s)/2)$. Let $T_1, \dots, T_n \subset [0,1]$ be disjoint, $\DQ_k$ discretized sets, and let $B \subset \RR^m$ be a $\DQ_{k+1}$ discretized set such that
    \begin{equation} \label{equation6091904232093}
        \#(\DQ_{k+1}(B)) \leq N_{k+1}^{s + \varepsilon}.
    \end{equation}
    Then there exists a constant $C(n,m,M) > 0$, and an integer constant $A(M) > 0$, such that if $A(M) \divides N_{k+1}$, and 
    \begin{equation} \label{equation19024u1298352389}
        N_{k+1} > C(n,m,M) \cdot M_{k+1}^{\frac{m}{m - (s + \varepsilon)}}.
    \end{equation}
    then there exists $\DQ_{k+1}$ discretized sets $S_1 \subset T_1$, \dots, $S_n \subset T_n$ such that
    \begin{enumerate}
        \item For any collection of $n$ distinct cubes $Q_1, \dots, Q_n \in \DQ_{k+1}(S_i)$,
        \[ Q_1 \times \dots \times Q_n \not \in \DQ_{k+1}(B). \]

        \item For each $i$, and for each $Q \in \DQ_k(T_i)$, there exists $\DR_Q \subset \DR_{k+1}(Q)$ such that
        \[ \#(\DR_Q) \geq \frac{\#(\DR_{k+1}(Q))}{A(M)}, \]
        and if $R \in \DR_{k+1}(Q)$,
        \[ \#(\DQ_{k+1}(R \cap S_i)) = \begin{cases} 1 &: R \in \DR_Q, \\ 0 &: R \not \in \DR_Q. \end{cases} \]
    \end{enumerate}
\end{theorem}
\begin{proof}
    For each $i \not \in \{ i_1, \dots, i_m \}$, there are rational numbers $a_{ij} = p_{ij}/q_{ij} \in \mathbf{Q}$ such that $M(e_i) = \sum a_{ij} e_j$. Set $A(M) = \prod_{ij} q_{ij}$. For each interval $R \in \DR_{k+1}(T_i)$, we let
    \[ a(R) \in \{ 0, \dots, N_1 \dots N_k M_{k+1} - 1 \} \]
    be the unique integer such that
    \[ R = \left[ \frac{a(R)}{N_1 \dots N_k M_{k+1}}, \frac{a(R) + 1}{N_1 \dots N_k M_{k+1}} \right]. \]
    Let $X \in \{ 0, \dots, N_{k+1}/M_{k+1} - 1 \}^m$. For each $1 \leq j \leq m$, define
    \[ S_{i_j}(X) = \bigcup_{R \in \DR_{k+1}(T_{i_j})} \left[ \frac{a(R)}{N_1 \dots N_k M_{k+1}} + \frac{X_j}{N_1 \dots N_{k+1}}, \frac{a(R)}{N_1 \dots N_k M_{k+1}} + \frac{X_j + 1}{N_1 \dots N_{k+1}} \right]. \]
    For $i \not \in \{ i_1, \dots, i_m \}$, define
    \[ S_i(X) = \bigcup_{\substack{R \in \DR_{k+1}(T_i)\\ \prod q_{ij} \divides a(R)}} \left[ \frac{a(R)}{N_1 \dots N_k M_{k+1}}, \frac{a(R)}{N_1 \dots N_k M_{k+1}} + \frac{1}{N_1 \dots N_{k+1}} \right] \]
    For each $i$, we let $\mathcal{S}_i(X)$ denote the set of startpoints to intervals in $S_i$. Then
    \[ \mathcal{S}_{i_j}(X) \subset \frac{\ZZ}{N_1 \dots N_k M_{k+1}} + \frac{X_j}{N_1 \dots N_{k+1}} \]
    and for $i \not \in \{ i_1, \dots, i_m \}$,
    \[ \mathcal{S}_i(X) \subset \frac{\prod q_{ij} \ZZ}{N_1 \dots N_k M_{k+1}}. \]
    It therefore follows that if
    \[ \mathcal{A}(X) = M(\mathcal{S}_1(X) \times \dots \times \mathcal{S}_n(X)), \]
    then
    \[ \mathcal{A}(X) \subset \frac{\ZZ^m}{N_1 \dots N_k M_{k+1}} + \frac{X}{N_1 \dots N_{k+1}}. \]
    In particular, if $X \neq X'$, $\mathcal{A}(X)$ and $\mathcal{A}(X')$ are disjoint. Equation \eqref{equation19024u1298352389} implies there is a constant $C(n,m,M)$, such that
    \begin{equation} \label{equation69129319031209}
    \begin{split}
        \#& \left\{ n \in \mathbf{Z}^m : d \left( \frac{n}{N_1 \dots N_{k+1}}, B \right) \leq \frac{2}{\sqrt{d} \cdot \| M \|} \frac{1}{N_1 \dots N_{k+1}} \right\}\\
        &\ \ \ \ \ \ \ \ \ \ \ \ \ \ \ \ \ \ \ \ \ \ \ \ \ \ \ \ \ \ \ \ \ \ \leq C(n,m,M)^{m - (s + \varepsilon)} \cdot N_{k+1}^{s + \varepsilon}.
    \end{split}
    \end{equation}
    Applying the pigeonhole principle, \eqref{equation6091904232093}, and \eqref{equation69129319031209}, there exists some value $X_0$ such that
    \begin{align*}
        \# \left\{ n \in \mathbf{A}(X_0) : d(n,B) \leq \frac{2}{\sqrt{d} \cdot \| M \|} \frac{1}{N_1 \dots N_{k+1}} \right\} &\leq \frac{C(n,m,M)^{m - (s + \varepsilon)} \cdot N_{k+1}^{s + \varepsilon}}{(N_{k+1}/M_{k+1})^m}\\
        &\leq \frac{C(n,m,M)^{m - (s + \varepsilon)} \cdot M_{k+1}^m}{N_{k+1}^{m - (s + \varepsilon)}} < 1.
    \end{align*}
    In particular, this set is actually empty. But this means that the set
    \[ M(S_1(X_0) \times \dots \times S_n(X_0)) \]
    is disjoint from $B$. Taking $S_i = S_i(X_0)$ for each $i$ completes the proof.
\end{proof}

Before we move on, consider one application of Theorem \ref{theorem9063909014901}, which gives an extension of Theorem \ref{sumset-application} to arbitrarily large sums.

\begin{theorem}
    Let $Y \subset \RR$ be a countable union of bounded sets with lower Minkowski dimension at most $t$. Then there exists a set $X \subset \RR$ with Hausdorff dimension at least $1 - t$ such that for any integer $n > 0$, for any $a_1, \dots, a_n \in \QQ$, and for any $x_1, \dots, x_n \in X$,
    \[ (a_1X + \dots + a_n X) \cap Y \subset (0). \]
\end{theorem}
\begin{proof}
    Let $Y = \bigcup_{i = 1}^\infty Y_i$, where each $Y_i$ has lower Minkowski dimension at most $t$. For each $n$, $i$, and $a = (a_1, \dots, a_n) \in \QQ^n$ with $a \neq 0$, let
    \[ \C_{n,a,i} = \{ (x_1, \dots, x_n) \in \C^n : a_1x_1 + \dots + a_nx_n \in Y_i \}, \]
    and let $\C = \bigcup \C_{n,a,i}$. Let $T_{n,a}(x_1,\dots,x_n)$ be the linear map given by
    \[ T_{n,a}(x_1,\dots,x_n) = a_1x_1 + \dots + a_nx_n. \]
    Then $T_{n,a}$ is nonzero, and $T_{n,a}(\C_{n,i,a})$ how lower Minkowski dimension at most $t$. Applying Theorem \ref{theorem9063909014901}, we obtain a set $X \subset [0,1]$ avoiding $\C$ with Hausdorff dimension at least $1 - t$.

    We prove $X$ satisfies the conclusions of this theorem by induction on $n$. Consider the case $n = 1$, and fix $a \in \QQ$. If $a \neq 0$, then because $X$ avoids $\C_{n,a,i}$ for each $i$, if $x \in X$, $ax \not \in Y$, so $aX \cap Y = \emptyset$. If $a = 0$, then $aX = 0$, so $(aX) \cap Y \subset (0)$.

    In general, consider $a = (a_1, \dots, a_{n+1}) \in \QQ^{n+1}$. If $a \neq 0$, then because $X$ avoids $\C_{n,a,i}$ for each $i$, we know if $x_1, \dots, x_{n+1} \in X$ are distinct, then $a_1 x_1 + \dots + a_{n+1} x_{n+1} \not \in Y$. If the values $x_1, \dots, x_{n+1} \in X$ are not distinct, then by rearranging both the values $\{ x_i \}$ and $\{ a_i \}$, we may without loss of generality assume that $x_n = x_{n+1}$. Then
    \begin{align*}
        a_1 x_1 + \dots + a_{n+1} x_{n+1} &= a_1 x_1 + \dots + a_{n-1} x_{n-1} + (a_n + a_{n+1}) x_n\\
        &\subset (a_1 X + \dots + a_{n-1} X + (a_n + a_{n+1}) X).
    \end{align*}
    By induction,
    \[ (a_1 X + \dots + a_{n-1} X + (a_n + a_{n+1}) X) \cap Y \subset (0), \]
    so we conclude that either $a_1 x_1 + \dots + a_{n+1} x_{n+1} \not \in Y$, or $a_1x_1 + \dots + a_{n+1} x_{n+1} = 0$. The only remaining case we have not covered is if $a \in \QQ^{n+1}$ is equal to zero. But in this case,
    \[ (a_1 X + \dots + a_n X) = (0 + \dots + 0) = 0, \]
    and so it is trivial that $(a_1 X + \dots + a_n X) \cap Y \subset (0)$.
\end{proof}

\section{Fourier Dimension}

Recently, there has been much interest in determining whether sets with large Fourier dimension can avoid configurations. Results published recently in the literature include \cite{PramanikLaba} and \cite{Shmerkin}. In this Section, we attempt to modify the procedure of Theorem \ref{mainTheorem} to obtain a set with large Fourier dimension. We obtain such a result, though with a suboptimal dimension to what we expect from Theorem \ref{mainTheorem}, and only holds in the setting where $d = 1$. We are currently researching methods to resolve the deficiencies in this method.

\begin{theorem} \label{FourierTheorem}
    Suppose $\C$ is a configuration on $\RR$, formed from the countable union of bounded sets, each with lower Minkowski dimension at most $s$. Then there exists a set $X \subset [0,1]$ with Fourier dimension at least $(n - s)/n$ avoiding $\C$.
\end{theorem}

We begin with a lemma which uses the Poisson summation theorem to restrict the analysis of the Fourier decay of the probability measures we study to the analysis of frequencies in $\ZZ$.

\begin{lemma} \label{discretefouriermeasures}
    Fix $s \in [0,d]$. Suppose $\mu$ is a compactly supported finite Borel measure on $\RR^d$. Then there exists a constant $A \geq 1$, depending only on the dimension of $d$ and the radius of the support of $\mu$, such that
    \[ \sup_{\xi \in \RR^d} |\xi|^{s/2} |\widehat{\mu}(\xi)| \leq 1 + A \left( \sup_{m \in \ZZ^d} |m|^{s/2} |\widehat{\mu}(m)| \right). \]
\end{lemma}
\begin{proof}
    Without loss of generality, we may assume that $\mu$ is supported on a compact subset of $[1/3,2/3)^d$, since every compactly supported measure is a finite sum of translates of measures of this form. Let
    \[ C = \sup_{m \in \ZZ^d} |m|^{s/2} |\widehat{\mu}(m)|, \]
    which we may assume, without loss of generality, to be finite. Consider the distribution $\Lambda = \sum_{m \in \mathbf{Z}^d} \delta_m$, where $\delta_m$ is the Dirac delta distribution at $m$. Then the Poisson summation formula says that the Fourier transform of $\Lambda$ is itself. If $\psi \in C_c(\RR^d)$ is a bump function supported on $[0,1)^d$, with $\psi(x) = 1$ for $x \in [1/3,2/3)^d$, then $\mu = \psi (\Lambda * \mu)$, so
    \begin{equation} \label{mubounded}
    \begin{split}
        |\widehat{\mu}(\xi)| &= \left| \left[ \widehat{\psi} * (\Lambda \widehat{\mu}) \right](\xi) \right|\\
        &= \left| \sum_{m \in \mathbf{Z}^d} \widehat{\mu}(m)(\widehat{\psi} * \delta_m)(\xi) \right|\\
        &= \left| \sum_{m \in \mathbf{Z}^d} \widehat{\mu}(m) \widehat{\psi}(\xi - m) \right|.
%       &\lesssim \sum_{n \in \mathbf{Z}^d} |\widehat{\mu}(n)| \prod_{i = 1}^d \frac{1}{1 + |n_i - \xi_i|}
    \end{split}
    \end{equation}
    Since $\psi$ is smooth, we know that for all $\eta \in \RR^d$, $|\widehat{\psi}(\eta)| \lesssim 1/|\eta|^{d+1}$. If we perform a dyadic decomposition, we find
    \begin{equation}
        \label{calculation1}
    \begin{split}
        \sum_{1 \leq |m - \xi| \leq |\xi|/2} |\widehat{\mu}(m)| |\widehat{\psi}(\xi - m)| &\leq C \sum_{1 \leq |m - \xi| \leq |\xi|/2} |\xi|^{-s/2} |\widehat{\psi}(\xi - m)|\\
        &\lesssim C \sum_{k = 1}^{\log |\xi|} \sum_{\frac{|\xi|}{2^{k+1}} \leq |m - \xi| \leq \frac{|\xi|}{2^{k}}} |\xi|^{-s/2} \left( 2^k/|\xi| \right)^{d+1}\\
        &\lesssim C \sum_{k = 1}^{\log |\xi|} |\xi|^{-s/2} (2^k / |\xi| ) \lesssim C |\xi|^{-s/2}.
    \end{split}
    \end{equation}
    There are $O_d(1)$ points $m \in \mathbf{Z}^d$ with $|m - \xi| \leq 1$, so if $|\xi| \geq 2$,
    \begin{equation} \label{calculation2}
        \sum_{|m - \xi| \leq 1} |\widehat{\mu}(m)| |\widehat{\psi}(m - \xi)| \lesssim C |\xi|^{-s/2}.
    \end{equation}
    We can also perform another dyadic decomposition, using the fact that for all $\eta \in \RR^d$, $|\widehat{\psi}(\eta)| \lesssim 1/|\eta|^{2d}$, to find that
    \begin{equation} \label{calculation3}
    \begin{split}
        \sum_{|m - \xi| \geq |\xi|/2} |\widehat{\mu}(m)| |\widehat{\psi}(m - \xi)| &\lesssim \sum_{k = 0}^\infty \sum_{|\xi| 2^{k-1} \leq |m - \xi| \leq |\xi| 2^k} \frac{|\widehat{\mu}(m)|}{|\xi|^{2d} 2^{2dk}}\\
        &\lesssim C \sum_{k = 0}^\infty |\xi|^{-d} 2^{-dk} \lesssim C |\xi|^{-d}.
    \end{split}
    \end{equation}
    Combining \eqref{calculation1}, \eqref{calculation2}, and \eqref{calculation3} with \eqref{mubounded}, we conclude that there exists a constant $A \geq 1$ depending only on the dimension $d$ such that if $|\xi| \geq 2$,
    \begin{equation} \label{endequation53}
        |\widehat{\mu}(\xi)| \leq A \cdot C \cdot |\xi|^{-s/2}.
    \end{equation}
    Since $\| \widehat{\mu} \|_{L^\infty(\RR^d)} \leq 1$, \eqref{endequation53} actually holds for all $\xi \in \RR^d$, provided $C \geq 1$.
\end{proof}

Our goal now is to carefully modify the discrete selection strategy and discretized probability measures we use to obtain have sharp control over the Fourier transform of these measures at each scale of our construction. In particular, we aim to obtain a sequence of pairs $\{ (\mu_k, X_k) \}$, where $\{ X_k \}$ is a decreasing family of subsets of $\RR$ such that the set $X = \bigcap X_k$ is configuration avoiding, and $\mu_k$ is a probability measure supported on $X_k$ such that for each integer $k$ and $\varepsilon > 0$, there exists a constant $C_{k,\varepsilon} > 0$ such that
\begin{equation} \label{equation12901909}
    \sup_{m \in \ZZ^d} |m|^{\frac{n - s}{2} - \varepsilon} |\widehat{\mu_{k+1}}(m) - \widehat{\mu_k}(m)| \leq C_{k,\varepsilon},
\end{equation}
and for each $\varepsilon > 0$, $\sum_{k = 0}^\infty C_{k,\varepsilon} < \infty$. We will initialize this sequence by letting $X_0 = [0,1]$, and let $\mu_0$ be the Lebesgue measure restricted to $X_0$. It then follows that if $\mu$ is a weak limit of some subsequence of the measures $\mu_k$, then $\widehat{\mu}$ is the pointwise limit of some subsequence of $\widehat{\mu_k}$. And so for each $\varepsilon > 0$,
\begin{align*}
    \sup_{m \in \ZZ} |m|^{\frac{n-s}{2} - \varepsilon} |\widehat{\mu}(m)| &\leq \sup_{m \in \ZZ} |m|^{\frac{n-s}{2} - \varepsilon} |\widehat{\mu_0}(m)| + \sum_{k = 0}^\infty \sup_{m \in \ZZ} |m|^{\frac{n-s}{2} - \varepsilon} |\widehat{\mu_{k+1}}(m) - \widehat{\mu_k}(m)|\\
    &\lesssim 1 + \sum_{k = 0}^\infty C_{k,\varepsilon} < \infty.
\end{align*}
Combined with Lemma \ref{discretefouriermeasures}, this implies that $X$, upon which $\mu$ is supported, has the required Fourier dimension. A key strategy to obtaining bounds of the form \eqref{equation12901909}, in light of our random selection strategy, is to obtain high probability bounds using Hoeffding's inequality.

\begin{theorem}[Hoeffding's Inequality]
    Let $\{ X_i \}$ be an independent family of $N$ mean-zero random variables, and let $A > 0$ be a constant such that $|X_i| \leq A$ almost surely for each $i$. Then for each $t > 0$,
    \[ \PP \left( \left| \frac{1}{N} \sum_{i = 1}^N X_i \right| \geq t \right) \leq 2 \exp \left( (N/A^2) \cdot (- t^2) \right). \]
\end{theorem}

As with Theorem \ref{mainTheorem}, we perform a multi-scale analysis, using the notations introduced in Section \ref{sec:Dyadics}. Lemma \ref{discretefouriermeasures} implies that we only need control over integer-valued frequencies. The discretized measures $\{ \nu_k \}$ we select are, for each $k$, a sum of point mass distributions at the points $\ZZ/N_1 \dots N_k$. Therefore, $\widehat{\nu_k}$ will be $N_1 \dots N_k$ periodic, in the sense that for any $m \in (N_1 \dots N_k) \ZZ$ and $\xi \in \RR$, $\widehat{\nu_k}(\xi + m) = \widehat{\nu_k}(\xi)$. Since we are only concerned with integer valued frequencies, it will therefore suffice to control the Fourier transform of $\nu_k$ on frequencies lying in $\{ 1, \dots, N_1 \dots N_k \}$.

In the discrete lemma below, we rely on a variant of the proof strategy of Theorem 2.1 of \cite{Shmerkin}, but modified so that we can allow the branching factors $\{ N_k \}$ to increase arbitrarily fast. For each $\DQ_k$ discretized set $E \subset [0,1]$, we define a probability measure
\[ \nu_E = \frac{1}{\#(\DQ_k(E))} \sum_{Q \in \DQ_k(E)} \delta(a(Q)), \]
where for each $x \in \RR$, $\delta(x)$ is the Dirac delta measure at $x$, and for each $Q \in \DQ_k$, $a(Q)$ is the startpoint of the interval $Q$. Also, for each $k$, we define a probability measure
\[ \eta_k = \frac{1}{N_k} \sum_{i_1, \dots, i_d = 0}^{N_k - 1} \delta \left( \frac{i}{N_1 \dots N_k} \right).  \]
The purpose of introducing $\eta_k$ is so that, given a measure $\mu$ which is a sum of point mass distributions in $\ZZ/N_1 \dots N_k$, the probability measure $\mu * \eta_{k+1}$ is a sum of point mass distributions in $\ZZ/N_1 \dots N_{k+1}$, uniformly distributed at the scale $1/N_1 \dots N_{k+1}$.

%Our goal now is now to carefully modify the discrete selection strategy and discretized probability measures we use, so that with high probability, the measures have the appropriate Fourier decay for the Fourier dimension bound we wish to obtain. Surprisingly, here we only need to perform a single scale analysis with the family of cubes $\DQ^d$, rather than a multi scale analysis involving the cubes $\DQ^d$ and $\DR^d$ as in Chapter \ref{ch:RoughSets}.

\begin{lemma} \label{discreteFourierBuildingBlock}
    Fix $s \in [1,dn)$, and $\varepsilon \in [0,(n-s)/4)$. Let $T \subset \RR$ be a nonempty, $\DQ_k$ discretized set, and let $B \subset \RR^n$ be a nonempty $\DQ_{k+1}$ discretized set such that
    \begin{equation} \label{equation982589128942189}
    \begin{split}
        \#(\DQ_{k+1}(B)) \leq N_{k+1}^{s + \varepsilon}.
    \end{split}
    \end{equation}
    %  \leq N_{k+1}^d
    %
    Provided that
    \begin{equation} \label{equation5523786128439}
        M_{k+1} \leq N_{k+1}^{\frac{n-s-2\varepsilon}{n}} \leq 2 M_{k+1},
    \end{equation}
    %
    %\begin{equation} \label{equation5523786128439}
    %    M_{k+1}^{\frac{n}{n - s - 2\varepsilon}} \leq N_{k+1} \leq 2 M_{k+1}^{\frac{n}{n - s - 2\varepsilon}},
    %\end{equation}
    %
    \begin{equation} \label{equation189248914891}
        \quad N_{k+1} \geq 3^{1/\varepsilon},
    \end{equation}
    \begin{equation} \label{equation8941894189238912}
        N_{k+1} \geq \exp \left( \left( \frac{4n}{n-s} \right)^4 N_1 \dots N_k \right),
    \end{equation}
    and
    \begin{equation} \label{equation77871247817841278}
        N_{k+1} \geq (1/\varepsilon)^{1/\varepsilon},
    \end{equation}
    there exists a universal constant $A(n,s)$ and a $\DQ_{k+1}$ discretized set $S \subset T$, satisfying the following properties:
    \begin{enumerate}
        \item[(A)] For any collection of $n$ distinct cubes $Q_1, \dots, Q_n \in \DQ_{k+1}(S)$,
        \[ Q_1 \times \dots \times Q_n \not \in \DQ_{k+1}(B). \]

        \item[(B)] For any $m \in \ZZ$,
        \[ |\widehat{\nu_S}(m) - \widehat{\eta_{k+1}}(m) \widehat{\nu_T}(m)| \leq A(n,s) \cdot (N_1 \dots N_{k+1})^{-\frac{n - s}{2n} + 2\varepsilon}. \]
    \end{enumerate}
\end{lemma}
\begin{proof}
    For each $R \in \DR_{k+1}(T)$, let $Q_R$ be randomly selected from $\DQ_{k+1}(R)$, such that the collection $\{ Q_R \}$ forms an independent family of random variables. Then, set $S = \bigcup \{ Q_R: R \in \DR_{k+1}(T) \}$. We then have
    \begin{equation} \label{equation6900921094190290}
        \#(\DQ_{k+1}(S)) = \#(\DR_{k+1}(T)) = M_{k+1} \cdot \DQ_k(T).
    \end{equation}
    Without loss of generality, removing cubes from $B$ if necessary, we may assume that for every cube $Q_1 \times \dots \times Q_n \in \DQ_{k+1}(B)$, the values $Q_1, \dots, Q_n$ occur in distinct intervals in $\DR_{k+1}(T)$. In particular, given any such cube, just as in Lemma \ref{discretelemma}, we have
    \begin{equation} \label{equation12043910293120909}
        \mathbf{P}(Q_1 \times \dots Q_n \in \DQ_{k+1}(S^n)) = (M_{k+1}/N_{k+1})^n.
    \end{equation}
    Thus \eqref{equation982589128942189}, \eqref{equation5523786128439}, and \eqref{equation12043910293120909} imply
    \begin{equation} \label{equation999992482}
        \mathbf{E} \left[ \#(\DQ_{k+1}(B) \cap \DQ_{k+1}(S^n)) \right] \leq M_{k+1}^n/N_{k+1}^{n - (s + \varepsilon)} \leq 1/N_{k+1}^\varepsilon.
    \end{equation}
    Markov's inequality, together with \eqref{equation189248914891} and \eqref{equation999992482}, imply
    \begin{equation} \label{fourierdim2}
    \begin{split}
        \mathbf{P}(\DQ_{k+1}(B) \cap \DQ_{k+1}(S^n) \neq \emptyset) &= \mathbf{P}(\# (\DQ_{k+1}(B) \cap \DQ_{k+1}(S^n)) \geq 1)\\
        &\leq 1/N_{k+1}^\varepsilon \leq 1/3.
    \end{split}
    \end{equation}
    Thus $\DQ_{k+1}(S^n)$ is disjoint from $\DQ_{k+1}(B)$ with high probability.

    Now we analyze the Fourier transform of the measures $\nu_S$. For each cube $R \in \DR_{k+1}(T)$, and for each $m \in \ZZ$, let
    \[ A_R(m) = e^{\frac{-2 \pi i m \cdot a(Q_R)}{N_1 \dots N_{k+1}}} - \frac{1}{N_{k+1}} \sum_{l = 0}^{N_{k+1} - 1} e^{\frac{-2 \pi i m \cdot [N_{k+1} a(Q) + l]}{N_1 \dots N_{k+1}}}. \]
    Then $\EE[A_R(m)] = 0$, $|A_R(m)| \leq 2$ for each $m$, and
    \[ \widehat{\nu_S}(m) - \widehat{\eta_{k+1}}(m) \widehat{\nu_T}(m) = \frac{1}{\#(\DR_{k+1}(T))} \sum_{R \in \DR_{k+1}(T)} A_R(m). \]
    Fix a particular value of $m$. Since the random variables $\{ A_R(m) : R \in \DR_{k+1}(T) \}$ are bounded and independent from one another, we can apply Hoeffding's inequality with \eqref{equation6900921094190290} to conclude that for each $t > 0$,
    \begin{equation} \label{equation5551902402919120}
    \begin{split}
        \PP \left( |\widehat{\nu_S}(m) - \widehat{\eta_{k+1}}(m) \widehat{\nu_T}(m)| \geq t \right) &\leq 2 \exp \left( \frac{- \#(\DR_{k+1}(T)) t^2}{4} \right)\\
        &= 2 \exp \left( \frac{- \#(\DQ_k(T)) M_{k+1} t^2}{4} \right).
    \end{split}
    \end{equation}
    The function $\widehat{\nu_S} - \widehat{\eta_{k+1}} \widehat{\nu_T}$ is $N_1 \dots N_{k+1}$ periodic. Thus, to uniformly bound this function, we need only bound the function over $N_1 \dots N_{k+1}$ values of $m$. Applying a union bound with \eqref{equation5551902402919120}, we conclude
    \begin{equation} \label{equation6662410242191209}
        \PP \left( \| \widehat{\nu_S} - \widehat{\eta_{k+1}} \widehat{\nu_T} \|_{L^\infty(\ZZ)} \geq t \right) \leq 2 N_1 \dots N_{k+1} \exp \left( \frac{- \#(\DQ_k(T)) M_{k+1} t^2}{4} \right).
    \end{equation}
    In particular, \eqref{equation5523786128439}, applied to \eqref{equation6662410242191209}, shows
    \begin{align*}
        \PP & \left( \| \widehat{\nu_S} - \widehat{\eta_{k+1}} \widehat{\nu_T} \|_{L^\infty(\ZZ)} \geq (N_1 \dots N_k M_{k+1})^{-1/2} \log(M_{k+1}) \right)\\
        &\ \ \ \ \leq 2N_1 \dots N_{k+1} \exp \left( - \frac{\#(\DQ_k(T)) \log(M_{k+1})^2}{4 N_1 \dots N_k} \right)\\
        &\ \ \ \ = 2 N_1 \dots N_k \exp \left( \log(N_{k+1}) - \frac{\log(M_{k+1})^2}{4 N_1 \dots N_k} \right)\\
        %&\ \ \ \ \leq 2 N_1 \dots N_k \exp \left( \log(N_{k+1}) - \frac{\log \left( N_{k+1}^{\frac{n-s-2\varepsilon}{2n}}/2 \right)^2}{4 N_1 \dots N_k} \right)\\
        &\ \ \ \ \leq 2 N_1 \dots N_k \exp \left( \log(N_{k+1}) - \left[ \left( \frac{n - s}{4n} \right) \log(N_{k+1}) - \log(2) \right]^2 \frac{1}{N_1 \dots N_k} \right).
    \end{align*}
    % A = 2N_1 ... N_k
    % D = 1/N_1 ... N_k
    % B = (n-s/4n)
    % C = log(2)
    %
    % A e(X - D (BX + C)^2) <= 3
    % X - D(BX + C)^2 <= log(3/A)
    % X - (B^2D)X^2 - 2BCDX - C^2D <= log(3/A)
    % (B^2D) X^2 + (2BCD - 1)X + [log(3/A) - C^2D] >= 0
    % X >= (1/2B^2D - C/B) + sqrt((2BCD - 1)^2 - 4(B^2D)(log(3/A) - C^2D))/2B^2D
    % X >= (4n/n-s)^4[N_1 ... N_k]
    % N_{k+1} \geq \exp \left( (4n/n-s)^4 [N_1 ... N_k] \right)
    % 
    Thus \eqref{equation8941894189238912} implies
    \begin{equation} \label{equation90120931902390190}
        \PP \left( \| \widehat{\nu_S} - \widehat{\eta_{k+1}} \widehat{\nu_T} \|_{L^\infty(\ZZ)} \geq (N_1 \dots N_k M_{k+1})^{-1/2} \log(M_{k+1}) \right) \leq 1/3.
    \end{equation}
    Taking a union bound over \eqref{fourierdim2} and \eqref{equation90120931902390190}, we conclude that there is a non-zero probability that the set $S$ satisfies Property (A), and
    \[ \| \widehat{\nu_S} - \widehat{\eta_{k+1}} \widehat{\nu_T} \|_{L^\infty(\ZZ)} \leq (N_1 \dots N_k M_{k+1})^{-1/2} \log(M_{k+1}). \]
    Since \eqref{equation77871247817841278} holds,
    % \log(x) \leq x^\varepsilon
    % 1/\varepsilon^{1/\varepsilon} \leq x
    % N_{k+1} \geq 1/\varepsilon^{1/\varepsilon}
    %
    \begin{align*}
        (N_1 \dots N_k M_{k+1})^{-1/2} \log(M_{k+1}) &\lesssim_{n,s} \log(N_{k+1}) (N_1 \dots N_{k+1})^{- \frac{n-s-2\varepsilon}{2n}}\\
        &\leq (N_1 \dots N_{k+1})^{- \frac{n-s}{2n} + 2\varepsilon}.
    \end{align*}
    Thus the set $S$ also satisfies Property (B) with an appropriately chosen constant $A(n,s)$.
\end{proof}

\begin{remark}
    Comparing the method of this lemma to that of Lemma \ref{discretelemma}, note that here we never needed to perform a non-random `deletion step' after the formation of our random set $S$. This is because \eqref{equation5523786128439} gives a larger gap between $N_k$ and $M_k$ than the gap provided by the analogous inequality \eqref{rBound} in Lemma \ref{discretelemma}, which results in the Fourier dimension bound in Theorem \ref{FourierTheorem} being smaller than the Hausdorff dimension bound in Theorem \ref{mainTheorem}. While performing non-random deletions of $O(M_{k+1}^d)$ intervals results in a negligible change to a Frostman-type bound like those dealt with in Chapter \ref{ch:RoughSets} once normalized. However, this many non-random deletions can cause a drastic shift in the Fourier transform of the associated measure on the set if these deletions occur non pseudorandomly. Improving Theorem \ref{FourierTheorem} thus requires a more subtle analysis of the deletions we must perform at each step.
\end{remark}

The construction of the set $X$ follows essentially the construction of the configuration avoiding set in Chapter \ref{ch:RoughSets}. We choose a decreasing sequence of parameters $\{ \varepsilon_k \}$ such that $\varepsilon_k < (n-s)/4$ for each $k$, as well as parameters $\{ N_k \}$ such that
\[ N_k \geq 3^{1/\varepsilon_k}, \]
\[ N_k \geq \exp \left( \left( \frac{4n}{n-s} \right)^4 N_1 \dots N_{k-1} \right), \]
\[ N_k \geq (1/\varepsilon_k)^{1/\varepsilon_k}, \]
\begin{equation} \label{equation13895891489132}
    N_k \geq (N_1 \dots N_{k-1})^{2/\varepsilon_k}.
\end{equation}
The choice of $N_k$ is also chosen sufficiently large that we can find a $\DQ_k$ discretized set $B_k$ such that
\[ \#(\DQ_k(B_k)) \leq (N_1 \dots N_k)^{s + \varepsilon_k/2} \leq N_k^{s + \varepsilon_k}, \]
and such that the collection $\{ B_k \}$ forms a strong cover of the configuration $\C$. We then choose a sequence $\{ M_k \}$  such that for each $k$,
\[ M_k \leq N_k^{\frac{n-s-2\varepsilon_k}{n}} \leq 2 M_{k+1}. \]
We also assume each $N_k$ and $M_k$ is a power of two, which means automatically that $M_k \divides N_k$ for each $k$.

Just as was done in Chapter \ref{ch:RoughSets}, this choice of parameters enables us to find a nested family of sets $\{ X_k \}$, obtained by setting $X_0 = [0,1]$, and letting $X_{k+1}$ be obtained from $X_k$ by applying Lemma \ref{discreteFourierBuildingBlock} with $\varepsilon = \varepsilon_{k+1}$, $T = X_k$, and $B = B_{k+1}$. We set $X = \bigcap X_k$. Since Property (A) of Lemma \eqref{discreteFourierBuildingBlock} is true at each step of the process, this is sufficient to guarantee that $X$ avoids the configuration $\C$. The remainder of this section is devoted to showing that Property (B) of Lemma \ref{discreteFourierBuildingBlock} is sufficient to obtain the Fourier dimension bound on $X$ guaranteed by Theorem \ref{FourierTheorem}.

Let $\nu_k = \nu_{X_k}$ for each $k$. Property (B) of Lemma \ref{discreteFourierBuildingBlock} implies that for each $k$,
\begin{equation} \label{equation77770123091293120}
    \left\| \widehat{\nu_{k+1}} - \widehat{\eta_{k+1}} \widehat{\nu_k} \right\|_{L^\infty(\ZZ)} \leq A(n,s) \cdot (N_1 \dots N_{k+1})^{- \frac{n-s}{2n} + 2\varepsilon_{k+1}}.
\end{equation}
We shall form a sequence of measures $\{ \mu_k \}$ by convolving the measures $\{ \nu_k \}$ with an appropriate family of mollifiers, which will be sufficient to obtain the required asymptotic bound.

\begin{lemma}
    There exists a sequence of probability measures $\{ \mu_k \}$, with $\mu_k$ supported on $X_k$ for each $k$, such that for each $\varepsilon > 0$,
    \[ \sup_{k > 0} \sup_{m \in \ZZ} |m|^{\frac{n-s}{2n} - \varepsilon} |\widehat{\mu_k}(m)| < \infty. \]
\end{lemma}
\begin{proof}
    For each $k$, let
    \[ \psi_k(x) = (N_1 \dots N_k) \cdot \mathbf{I}_{\left[ 0, \frac{1}{N_1 \dots N_k} \right]}. \]
    Then it is easy to calculate that
    \begin{equation} \label{equation901418294891481792}
        |\widehat{\psi_k}(m)| \lesssim \min \left( 1, \frac{N_1 \dots N_k}{|m|} \right).
    \end{equation}
    Note that the measures $\mu_k = \nu_k * \psi_k$ are still supported on $X_k$, and
    \[ \widehat{\mu_k}(\xi) = \widehat{\nu_k}(\xi) \widehat{\psi_k}(\xi). \]
    Also note that $\psi_k = \psi_{k+1} * \eta_{k+1}$. If $\varepsilon > 0$, then we can apply \eqref{equation77770123091293120} with \eqref{equation901418294891481792} to conclude
    \begin{equation} \label{equation6892489214781278}
    \begin{split}
        &|\widehat{\mu_{k+1}}(m) - \widehat{\mu_k}(m)|\\
        &\ \ \ \ = |\widehat{\psi_{k+1}}(m)| |\widehat{\nu_k}(m) - \widehat{\eta_{k+1}}(m) \widehat{\nu_k}(m)|\\
        &\ \ \ \ \lesssim \min \left( 1, \frac{N_1 \dots N_{k+1}}{|m|} \right) (N_1 \dots N_{k+1})^{-\frac{n-s}{2n} + 2\varepsilon_{k+1}}.\\
        &\ \ \ \ = \min \left( \frac{|m|^{\frac{n-s}{2n} - \varepsilon}}{(N_1 \dots N_{k+1})^{\frac{n-s}{2n} - 2\varepsilon_{k+1}}}, \frac{(N_1 \dots N_k)^{1 - \frac{n-s}{2n} + 2\varepsilon_{k+1}}}{|m|^{1 - \frac{n-s}{2n} + \varepsilon}} \right) |m|^{- \frac{n-s}{2n} + \varepsilon}.
    \end{split}
    \end{equation}
    The minima is maximized when $|m| = N_1 \dots N_{k+1}$, which gives
    \[ \min \left( \frac{|m|^{\frac{n-s}{2n} - \varepsilon}}{(N_1 \dots N_{k+1})^{\frac{n-s}{2n} - 2\varepsilon_{k+1}}}, \frac{(N_1 \dots N_k)^{1 - \frac{n-s}{2n} + 2\varepsilon_{k+1}}}{|m|^{1 - \frac{n-s}{2n} + \varepsilon}} \right) \leq (N_1 \dots N_{k+1})^{2\varepsilon_{k+1} - \varepsilon}. \]
    Thus, for all $k$, for all $m \in \ZZ$, and for all $\varepsilon > 0$,
    \begin{equation} \label{equation11020404120}
    \begin{split}
        |\widehat{\mu_{k+1}}(m) - \widehat{\mu_k}(m)| \lesssim \frac{(N_1 \dots N_{k+1})^{2\varepsilon_{k+1} - \varepsilon}}{|m|^{\frac{n-s}{2n} - \varepsilon}}.
    \end{split}
    \end{equation}
    For each $k$, let
    \[ A_k = \sup_{m \in \ZZ} |\widehat{\mu_k}(m)| |m|^{\frac{n-s}{2n} - \varepsilon}. \]
    Then \eqref{equation11020404120} implies that
    \[ A_{k+1} = A_k + O \left( (N_1 \dots N_{k+1})^{2\varepsilon_{k+1} - \varepsilon} \right). \]
    Thus for all $k > 0$,
    \[ A_k = O \left( \sum_{k = 1}^\infty (N_1 \dots N_k)^{2\varepsilon_{k+1} - \varepsilon} \right). \]
    Provided the sum on the right hand side converges for each $\varepsilon > 0$, this gives a uniform bound of $A_k$ in $k$ for each $\varepsilon > 0$, completing the proof. But for suitably large $k$, depending on $\varepsilon$, it is eventually true that $\varepsilon_{k+1} \leq \varepsilon/8$, and so
    \begin{align*}
        A_k &= O_\varepsilon(1) + \sum_{k = 1}^\infty (N_1 \dots N_k)^{-\varepsilon/4} = O_\varepsilon(1) + \sum_{k = 1}^\infty 2^{-k\varepsilon/4} = O_\varepsilon(1). \qedhere
    \end{align*}
\end{proof}

Just as for the sequence of measures in Theorem \ref{massdistributionprinciplelem}, the sequence $\{ \mu_k \}$ is a Cauchy sequence of probability measures, and therefore converges weakly to some measure $\mu$. Because for each $k$, $\mu_k$ is supported on $X_k$, $\mu$ is supported on $\bigcap X_k = X$. Furthermore, the Fourier transform of each $\mu_k$ converges pointwise to the Fourier transform of $\mu$. Thus we find that for each $\varepsilon > 0$,
\[ \sup_{m \in \ZZ} |m|^{\frac{n-s}{2n} - \varepsilon} |\widehat{\mu}(m)| \leq \sup_{k > 0} \sup_{m \in \ZZ} |m|^{\frac{n-s}{2n} - \varepsilon} |\widehat{\mu_k}(m)| < \infty. \]
Combined with Lemma \ref{discretefouriermeasures}, this implies $X$ has Fourier dimension $(nd - s)/n$.

% 3. Notes
% 4. Footnotes

% 5. Bibliography
\begin{singlespace}
\raggedright
\bibliographystyle{abbrvnat}
\bibliography{biblio}
\end{singlespace}

%\appendix
% 6. Appendices (including copies of all required UBC Research
% Ethics Board's Certificates of Approval)
% \include{reb-coa}	% pdfpages is useful here
%\include{appendix}

\nocite{*}
\backmatter
% 7. Index
% See the makeindex package: the following page provides a quick overview
% <http://www.image.ufl.edu/help/latex/latex_indexes.shtml>

\end{document}